\documentclass[12pt]{amsart}
\usepackage{amsfonts}
\usepackage{amsmath}
\usepackage{amssymb}
\usepackage{graphicx}
\usepackage{enumerate}
\usepackage{multicol}
\usepackage{mathrsfs}
\usepackage[usenames,dvipsnames]{pstricks}
\usepackage{epsfig}
\usepackage{pst-grad} % For gradients
\usepackage{pst-plot} % For axes
\usepackage[hmargin=3cm,vmargin=2.3cm]{geometry}
\usepackage[utf8]{inputenc}

\newtheorem{defi}{Definition}
\newtheorem{rem}{Remark} 
\newtheorem{prop}{Proposition}
\newtheorem{theorem}{Theorem}

\newtheorem{remark}{Remark}

\newtheorem{coro}{Corollary}
\newtheorem{lemma}{Lemma}

\DeclareMathOperator{\Id}{Id}
\DeclareMathOperator{\re}{Re}
\DeclareMathOperator{\im}{Im}

\DeclareMathOperator{\tr}{tr}

\DeclareMathOperator{\spn}{span}
\DeclareMathOperator{\diag}{diag}

\DeclareMathOperator{\Hd}{\mathcal{H}}
\DeclareMathOperator{\Vd}{\mathcal{V}}

\title{Sub-Riemannian geometry of Stiefel manifolds}
\author[Christian Autenried, Irina Markina]{ Christian Autenried, Irina Markina}

\address{Department of Mathematics, University of Bergen, Norway.}
\email{christian.autenried@math.uib.no}

\address{Department of Mathematics, University of Bergen, Norway.}
\email{irina.markina@math.uib.no}

\thanks{The authors are partially supported by the NFR-FRINAT grants \#204726/V30 and \#213440/BG.}



\begin{document}
\maketitle

\begin{abstract}
In the paper we consider the Stiefel manifold $V_{n,k}$ as a principal $U(k)$-bundle over the Grassmann manifold and study the cut locus from the unit element. We gave the complete description of this cut locus on $V_{n,1}$ and presented the sufficient condition on the general case. At the end, we study the complement to the cut locus of $V_{2k,k}$.
\end{abstract}

\section{Introduction}

A sub-Riemannian geometry is an abstract setting for study geometry with non-holonomic constraints. A sub-Riemannian manifold is a triplet $(Q,D,g_D)$, where $Q$ is a $C^{\infty}$-smooth manifold, $D$ is a smooth sub-bundle of the tangent bundle $TQ$ of the manifold $Q$ (or smooth distribution) and $g_D$ is a smoothly varying with respect to $q\in Q$ inner product $g_D(q)\colon D_q\times D_q\to \mathbb R$. The topic is actively developed last decades and as, now classical, sources we refer to~\cite{AgrSach,CDPT,LS,M,Str}.

One of the main objects of interest in sub-Riemannian geometry are normal and abnormal geodesics that are two different but not mutually disjoint families. The exponential map is not a local diffeomorphism anymore. Nevertheless, the singularities of the exponential map, as in the Riemannian geometry are closely related to the cut locus and failure of the optimality for geodesics. The cut locus in sub-Riemannian geometry is an object that is of big interest, but rather poorly studied. There exist very few results concerning the global and local structure of it and most of them restricted to low dimensional manifolds. The work~\cite{NaSa} studies the one dimensional Heisenberg group, and the results easily can be extended to higher dimensions. A full description of the global structure of the cut locus for the groups $SU(2)$, $SO(3)$, $SL(2)$, and lens spaces is given in~\cite{B}. For the groups $SO(3)$, $SL(2)$, and lens spaces the cut locus is a stratified set, whereas in $SU(2)$ it is a maximal circle $S^1$ without one point. The reader will find similar structures to those that obtained in the present work.  The global structure of the exponential map and the cut locus of the identity on the group $SE(2)$ is completely presented in~\cite{S}.
 
The nature of normal and abnormal geodesics and complexity of the cut locus structure in sub-Riemannian geometry on the example of the Martinet manifold is pointed out in the work~\cite{cut}.
The Martinet manifold is the smooth manifold $\mathbb{R}^3$ with smooth distribution spanned by vector fields 
$$ X=\frac{\partial}{\partial x} + \frac{1}{2}y^2\frac{\partial}{\partial t}, \qquad Y = \frac{\partial}{\partial y}$$
and an inner product, making $X,Y$ orthonormal. The cut locus in this case is the Martinet surface $y=0$ minus the abnormal geodesic $z=0$ inside of the surface~[Thm. 1.2, \cite{cut}]. The cut locus for contact manifolds were also studied in~\cite{AG}.

A progress in study of the cut locus of the identity on the sub-Lorentzian counterpart of one dimensional Heisenberg group can be found in~\cite{GR}.

In the present work we consider the Stiefel manifold $V_{n,k}$ as a principal $U(k)$-bundle with the Grassmann manifold as a base space. We completely describe the cut locus from the unit element for the case $V_{n,1}$. The technical difficulties and possible presents of abnormal geodesics did not allowed to extend this result to the general case $V_{n,k}$. Nevertheless, we present a partial description of the cut locus, that is to our knowledge almost unique example for manifolds of higher dimensions.

The structure of the work is the following. Section~\ref{sec:basic} collects the basic definitions that nowadays are standard in sub-Riemannian geometry, but sometimes fussy. In Section~\ref{VG} we define Grassmann and Stiefel manifolds embedded in $U(n)$, metric of constant bi-invariant type and normal geodesics based on the general theorem that can be found in~\cite{M}. In Section~\ref{sec:Vn1} we describe the cut locus for the equivalence class of the unit element on the principal $U(1)$-bundle structure on the Stiefel manifold $V_{n,1}$. Since the considered manifold is homogeneous it gives the structure of the cut locus for any point. Section~\ref{sec:general} is dedicate to the cut locus for the general case of the Stiefel manifold $V_{n,k}$ and $V_{2k,k}$. In Section~\ref{real} we briefly review some particular cases of the Stiefel manifold embedded in $SO(n)$.

%%%%%%%%%%%%%%%%%%%%%%%%%%%%%

\section{Basic definitions from sub-Riemannian geometry}\label{sec:basic}

%%%%%%%%%%%%%%%%%%%%%%%%%%%%%

We remind the necessary definitions and propositions based on~\cite{M}.
\begin{defi}
A sub-Riemannian manifold is a triplet $(Q,\Hd,\langle \cdot\, , \cdot \rangle)$, where $Q$ is a $C^{\infty}$-manifold, $\Hd$ is a vector subbundle of the tangent bundle $TQ$, and $\langle \cdot \,, \cdot \rangle$ is a fibre inner-product. The subbundle $\Hd$ is called horizontal and $\Hd_q$ is a horizontal space at a point $q\in Q$. The metric $\langle \cdot \,, \cdot \rangle_q\colon \Hd_q\times \Hd_q\to\mathbb R$, $q\in Q$ is called a sub-Riemannian metric, and the couple $(\Hd,\langle \cdot \,, \cdot \rangle)$ is a sub-Riemannian structure on $Q$.
\end{defi}

\begin{defi}
The horizontal subbundle $\Hd$ is called bracket generating if for every $q\in Q$ there exists $r(q) \in \mathbb{Z}^+$ s.t. 
\begin{equation*}
\Hd^{r(q)}=T_qQ, 
\end{equation*} 
where $\Hd^1:=\Hd$ and $\Hd^{r+1}:=[\Hd^r, \Hd]+\Hd^r$, $r \geq 1$.
\end{defi}

\begin{defi}
An absolutely continuous curve $\gamma \colon [0,T] \to Q$ is called horizontal if $\dot{\gamma}(t) \in \Hd_{\gamma(t)}$ almost everywhere.
\end{defi}

\begin{defi}\label{length}
We define the length $l:=l(\gamma)$ of an absolutely continuous horizontal curve $\gamma\colon[0,T]\to Q$ as in the Riemannian geometry:
$$l(\gamma):=\int_0^T{\Vert \dot{\gamma} \Vert } dt =\int_0^T\sqrt{\langle \dot{\gamma}(t) , \dot{\gamma}(t) \rangle }\,dt.$$
Introduce the function $d(q_0,q)$ for $q_0, q \in Q$ by 
$$d(q_0, q):=\inf_{\gamma}\{l(\gamma)\},$$
where the infimum is taken over all absolutely continuous horizontal curves that connect $q_0$ and $q$. If there is no horizontal curve
joining $q_0$ to $q$, then we declare $d(q_0,q)=\infty$.
\end{defi}

Recall the Chow-Rashevskii theorem~\cite{C, R} that gives a sufficient condition of the existence of horizontal curves.

\begin{theorem}
Let $Q$ be a connected manifold. If the horizontal subbundle $\Hd\subset TQ$ is bracket generating, then any two points in $Q$ can be joined by a horizontal curve.
\end{theorem}

It follows that if $\Hd$ is bracket generating on a connected manifold, then the function $d$ introduced in Definition~\ref{length} is finite and defines the distance between two points on the manifold, called Carnot-Carath\'eodory distance.

\begin{defi}
An absolutely continuous horizontal curve that realizes the distance between two points is called a minimizing geodesic.
\end{defi}

Let $Q$ be $n$-dimensional smooth manifold and $\Hd$ be a smooth horizontal subbundle such that $\dim\Hd_q=k\leq n$ for all $q\in Q$. Considering a neighborhood $U_q$ around $q \in Q$ such that the subbundle $\Hd$ is trivialized in $U_q$, one can find a local orthonormal basis $X_1, \dotso , X_k$ with respect to the sub-Riemannian metric $\langle \cdot\, , \cdot \rangle$. The associated sub-Riemannian metric Hamiltonian is given by 
\begin{eqnarray*}
H(p,\lambda)=\frac{1}{2}\sum_{m=1}^k{\lambda(X_m(p))^2},
\end{eqnarray*}
where $(p,\lambda) \in T^*U_q$. A {\it normal geodesic} is defined as the projection to $U_q\subset Q$ of the solution to the Hamiltonian system
\begin{equation*}
\dot{p}_i=\frac{\partial H}{\partial \lambda_i} \qquad
\dot{\lambda}_i = -\frac{ \partial H}{\partial p_i},
\end{equation*}
where $(p_i,\lambda_i)$ are the coordinates in $T^*U_q$. 
We note that the word ``normal'' appears due to the fact that in the sub-Riemannian geometry there is another type of geodesics, calling ``abnormal" arising from different type of Hamiltonian function. For a more detailed examination of abnormal geodesics we refer to~\cite{AS,AS1,BT,LS,M94}. The present work is mostly concerned with the normal geodesics, therefore we omit the detailed definition for abnormal ones.

Suppose two differentiable manifolds $Q$, $M$, and the submersion $\pi\colon Q \to M$ are given. The fibre through $q\in Q$ is the set $Q_m:=\pi^{-1}(m)$, 
$m=\pi(q)$, which is a submanifold according to the implicit function theorem. 
The differential $d_q\pi\colon T_qQ \to T_{\pi(q)}M$ of $\pi$ defines the vertical space $\mathcal{V}_q \subset T_qQ$ that is the tangent space to the fibre $Q_{\pi(q)}$ and it is written as $\mathcal{V}_q:=\ker(d_q\pi)=T_q(Q_m)$, where $\ker(d_q\pi)$ denotes the kernel of the linear map $d_q\pi$. It can be shown that $\Vd=\bigcup_{q\in Q}\Vd_q$ is a smooth subbundle of $TQ$ that is called vertical subbundle~\cite{M}.

\begin{defi}
An Ehresmann connection $($or connection$)$ for a submersion $\pi\colon Q \to M$ is a subbundle $\mathcal{H} \subset TQ$ that is everywhere transverse and
of complementary dimension to the vertical: $\mathcal{V}_q \oplus \mathcal{H}_q=T_qQ$. The space $\mathcal{H}_q$ is called horizontal
subspace of $T_qQ$.
\end{defi}

\begin{defi}
Let $\pi\colon Q \to M$ be a submersion with connection $\mathcal{H}$ and let $c\colon I \to M$ be a curve starting at $m \in M$. A curve $\gamma\colon I
\to Q$ is called a horizontal lift of the curve $c$ if $\gamma $ is tangent to $\mathcal{H}$ and projects to $c$, i.e. $\dot{\gamma}(t) \in \mathcal{H}_{\gamma (t)}$ and $\pi \circ \gamma(t) = c(t)$
for all $t \in I$.
\end{defi} 
There are different ways to introduce a sub-Riemannian structure on $Q$. In the sequel we describe two of them and indicate when they coincide. 

Assuming that $Q$ is a Riemannian manifold in the submersion $\pi\colon Q \to M$, we can use its Riemannian metric to define the orthogonal complement $\Hd_q$ of the vertical space $\Vd_q$ at each point $q \in Q$. Then $\mathcal{H}$ is a connection and the restriction of the Riemannian metric to $\mathcal{H}$ defines a sub-Riemannian metric on $Q$. 

Assume that the manifold $M$ is endowed with a Riemannian metric and the submersion $\pi \colon Q \to M$ has a connection $\mathcal{H}$.
Since $\mathcal{V}_q=\ker(d_q\pi)$ and $\im(d_q\pi\vert_{\mathcal{H}_q})=\im(d_q\pi)=T_{\pi(q)}M$, it follows that $d_q\pi\vert_{\Hd_q}$ is a linear
isomorphism from $\mathcal{H}_q$ to $T_{\pi(q)}M$. By pulling back the Riemannian metric on $M$ to $Q$,  we obtain a sub-Riemannian metric on $Q$ with underlying subbundle $\mathcal{H}$. This sub-Riemannian metric is said to be induced by the
connection $\mathcal{H}$ on $Q$ and the Riemannian metric on $M$. 

Suppose $Q$ and $M$ are smooth Riemannian manifolds and a submersion $\pi\colon Q \to M$ is given. Let  $\mathcal{H}_q$ be orthogonal
complement to the vertical $\mathcal{V}_q$ at every $q \in Q$. Two ways of inducing a sub-Riemanian metric on $Q$, by
restricting the Riemannian metric of $Q$ or by pulling back the Riemannian metric on $M$ using $d\pi$, coincide if $d_q\pi$ restricts to a linear isometry $\mathcal{H}_q \to T_{\pi(q)}M$ for all $q \in Q$.
\begin{defi}\label{isometry}
Let $Q$ and $M$ be Riemannian manifolds and let $\pi\colon Q \to M$ be a submersion. Let $\mathcal{V}_q \subset T_qQ$ denote the vertical
subspace at $q\in Q$ and $\Hd_q:=\mathcal{V}_q^{\perp}$ be its orthogonal complement. If $d\pi\colon TQ \to TM$ restricts to a linear isometry $\mathcal{H}_q
\to T_{\pi(q)}M$ for each $q \in Q$, then $\pi$ is called a Riemannian submersion.
\end{defi} 
Thus, Riemannian metrics on $Q$ and $M$ induce the same subriemannian structure on $Q$ if the submersion is Riemannian. 
\begin{defi}
A fibre bundle $\pi\colon Q \to M$ is a principal $G$-bundle if its fibre is a Lie group $G$ that acts freely and transitively on each fibre. 
\end{defi}
As a consequence we can identify $M$ with the quotient $Q/G$ of $Q$ by the group action of $G$. Furthermore, $\pi$ corresponds to the canonical projection to the
quotient.
\begin{defi}
A connection on $\pi\colon Q \to M$ is a principal $G$-bundle connection if the action of $G$ preserves the connection.
\end{defi}
We assume that the group acts on itself on the right $q \mapsto qg$, $q \in Q$, $g \in G$.

\begin{defi}
Let $Q \to M$ be a principal $G$-bundle with connection $\mathcal{H}$. A sub-Riemannian metric on $(Q,\mathcal{H})$ which is invariant under the action of $G$ is called a metric of bundle type.
\end{defi}
A sub-Riemannian metric which is induced from a $G$-invariant metric on $Q$ is an example of a metric of bundle type.

\begin{defi}
A bi-invariant Riemannian metric $\langle\cdot\,,\cdot\rangle$ on a differentiable manifold $Q$ with the Lie group $G$ acting on it is said to be of constant bi-invariant type if its inertia tensor
$\mathbb{I}_q\colon\mathfrak{g} \times \mathfrak{g} \to \mathbb{R}$ defined by $\mathbb{I}_q(\xi , \eta):=\langle \sigma_q \xi , \sigma_q\eta \rangle$ is
independent of $q \in Q$. Here
\begin{eqnarray*}
 \sigma_q\colon\mathfrak{g} &\to & T_qQ \\
\xi &\mapsto & \frac{d}{d \epsilon } \bigg| _{\epsilon=0} q \exp(\epsilon \xi).
\end{eqnarray*}
\end{defi}

\begin{defi}
Let $\pi \colon Q \to M$ be a principal $G$-bundle with a Riemannian metric of constant bi-invariant type and $\mathcal{H}$ be the induced connection. 
We define the $\mathfrak{g}$-valued connection one-form $A_q$ uniquely as the linear operator $A_q\colon T_qQ \to \mathfrak{g}$ which satisfies
following properties:
\begin{equation*}
\ker (A_q)=\mathcal{H}_q,\qquad A_q \circ \sigma_q =Id_{\mathfrak{g}},
\end{equation*}
where $Id_{\mathfrak{g}}$ is the identity map on $\mathfrak{g}$.
\end{defi}
The map $A\colon TQ \to \mathfrak{g}$ defines a $\mathfrak{g}$-valued connection tensor on $Q$.

\begin{theorem}~\cite{M}\label{Theorem}
Let $\pi\colon Q \to M$ be a principal $G$-bundle with a Riemannian metric of constant bi-invariant type. Let $\mathcal{H}$ be the induced connection, 
with $\mathfrak{g}$-valued connection tensor $A$. Let $\exp_R$ be the Riemannian exponential map, so that $\gamma_R(t)=\exp_R(tv)$ 
is the Riemannian geodesic through $q$ with initial velocity $v\in T_qQ$. Then any horizontal lift $\gamma$ of the projection $\pi\circ \gamma_R$ 
is a normal sub-Riemannian geodesic and is given by
$$\gamma(t)=\exp_R(tv)\exp_G(-tA(v)),$$
where $\exp_G\colon \mathfrak{g} \to G$ is the group $G$ exponential map. Moreover, all normal sub-Riemannian geodesics can be obtained in this way.
\end{theorem}

%%%%%%%%%%%%%%%%%%%%%%%%%%%%%%%%%%%%%%%%%%%%%%%%

\section{Stiefel and Grassmann manifolds embedded in $U(n)$}\label{VG}

%%%%%%%%%%%%%%%%%%%%%%%%%%%%%%%%%%%%%%%%%%%%%%%%

We use the following notations in the present section. Let $\mathbb{C}^n$ denote a $n$-dimensional complex vector space and $\mathbb{C}^{m \times n}$ the set of $(m \times n)$-matrices with complex entries. 
We want to apply Theorem~\ref{Theorem} for the submersion 
$\pi \colon V_{n,k}(\mathbb{C}^n) \to G_{n,k}(\mathbb{C}^n)$, where $V_{n,k}(\mathbb{C}^n)=V_{n.k}$ is the Stiefel manifold 
and $G_{n,k}(\mathbb{C}^n)=G_{n,k}$ is the Grassmann manifold for $n \in \mathbb{N}$ and $k \in \{1, \dotso ,n\}$. 

We start from the description of a general construction. Given a group $G$ with an invariant inner product on its Lie algebra $\mathfrak g$ and two subgroups $H,K \subset G$, we form the quotient spaces $G/H$ and $G/(H \times K)$. The submersion $G/H \rightarrow G/(H \times K)$ is a principal $K$-bundle, with Riemannian metrics on $G/H$ and $G/(H \times K)$ induced from the bi-invariant Riemannian metric on $G$ generated by an invariant inner product. The Riemannian metrics are induced by the projections $G\to G/H$ and $G\to G/(H \times K)$. Both manifolds in the submersion $G/H \rightarrow G/(H \times K)$ are homogeneous manifolds, where the group $G$ acts transitively. The induced Riemannian metric on $G/H$ is also bi-invariant under the action of the group $G$. The geodesics on $G/H$ are the projections from $G$ of one-parameter subgroups $\exp(t \xi )$ with $\xi$ orthogonal to the Lie algebra $\mathfrak h \subset \mathfrak g$ of $H$. We set $G=U(n)$, $H=U_n(n-k)$, $K=U_n(k)$, where 
$$U_n(k):=\left\{\begin{pmatrix} U_k & 0 \\ 0 & I_{n-k} \end{pmatrix} \Big\vert\ U_k \in U(k) \right\} \subset U(n)\quad\text{ and}
$$ 
$$U_n(n-k):=\left\{\begin{pmatrix} I_k & 0 \\ 0 & U_{n-k} \end{pmatrix} \Big\vert U_{n-k} \in U(n-k) \right\} \subset U(n).$$ 
Note that we use the notations $U_n(k)$ and $U_n(n-k)$ with the lower subscript in the current section to emphasise that the elements of these groups are written as $(n\times n)$-matrices. Then the quotient $G/H=U(n)/U_n(n-k)$ is isomorphic to the Stiefel manifold $V_{n,k}$ and $G/(H \times K)=U(n)/(U_n(n-k) \times U_n(k))$ is isomorphic to the Grassmann manifold $G_{n,k}$.

%%%%%%%%%%%%%%%%%%%%%%%%

\subsection{Unitary group and bi-invariant metric}

%%%%%%%%%%%%%%%%%%%%%%%%

Before we give detailed definition of Stiefel and Grassmann manifolds we remind that the unitary group $U(n)$ is a matrix Lie group, whose elements $X$ satisfy the condition
$$U(n)=\{X\in\mathbb{C}^{n\times n} \vert\ \  \bar{X}^TX=X\bar{X}^T=I_n\}.$$
Here $I_n$ is the unite $(n\times n)$-matrix and $\bar{X}^T$ is the complex conjugate and transposed of the matrix $X$. The Lie algebra $\mathfrak{u}(n)$ consists  of all skew-Hermitian matrices: 
$$\mathfrak{u}(n)=\{\mathcal X
\in \mathbb{C}^{n \times n} \vert\ \ \mathcal X=-\bar{\mathcal X}^{T}\}.
$$ 
We remind that a matrix $X\in U(n)$ is of full rank, its columns and rows are orthonormal with respect to the standard Hermitian product in $\mathbb{C}^n$ and that the main diagonal of the skew-Hermitian matrices are purely imaginary. Moreover, the Hermitian product in $\mathbb C^n$ is invariant under the action of $U(n)$, that particularly means that the orthogonality is preserved under this action. The Lie algebra $\mathfrak u(n)$ can be endowed with the inner product $(\mathcal X ,\mathcal Y)_{\mathfrak u(n)}-2n\tr(\mathcal X\mathcal Y)$, $\mathcal X,\mathcal Y\in \mathfrak u(n)$. 
Considering $U(n)$ as a smooth manifold, we denote its points by~$q$ and the metric at this point by $\langle\cdot\,,\cdot\rangle_{U(n)}(q)$ or, if it is clear from the context, simply by $g_q$. Then a left-invariant metric on $U(n)$ with respect to the group action of $U(n)$ on its Lie algebra is given by
$$
\begin{array}{lllll}
\langle\cdot\,,\cdot\rangle_{U(n)}(q) \colon & T_qU(n) \times T_qU(n)\cong & q\mathfrak{u}(n) \times q\mathfrak{u}(n) &\to &\mathbb{R}
\\
& & ( q\mathcal X\ ,\ q\mathcal Y )&\mapsto & -2n\tr(\mathcal X\mathcal Y)
\end{array}
$$ 
$q\in U(n)$, $\mathcal X,\mathcal Y \in \mathfrak{u}(n)$. 
This metric is actually bi-invariant, that follows from the observation that can be found, for instance, in~\cite{G} and~\cite{Mi}. It is stated as follows: Let $\mathfrak g$ be a Lie algebra of a Lie group $G$ endowed with an inner product~$( \cdot \,,\cdot )_{\mathfrak g}$. An inner product $( \cdot \,,\cdot )_{\mathfrak g}$ is called invariant if it is invariant under the adjoint action of $G$, i.e. 
$(q^{-1} \eta q , q^{-1} \xi q )_{\mathfrak g} = (\eta , \xi )_{\mathfrak g}$
for all $\eta , \xi \in \mathfrak g$ and $q\in G$. Then it is well known, see for instance~\cite{Knapp}, that an invariant inner product $( \cdot \,,\cdot )_{\mathfrak g}$ on a Lie algebra $\mathfrak g$ determines a bi-invariant metric on the group $G$ via 
$$ \langle \eta , \xi \rangle_{G}(q) := (q^{-1} \eta, q^{-1} \xi )_{\mathfrak g}=(\eta q^{-1} , \xi q^{-1} )_{\mathfrak g}$$
for all $\eta , \xi \in T_qG$. 

We only need to check that 
the inner product $(\mathcal X ,\mathcal Y)_{\mathfrak u(n)}=-2n\tr (\mathcal X \mathcal Y)$ on $\mathfrak u (n)$ is invariant. Indeed, 
\begin{eqnarray*}
(q^{-1}\mathcal X q , q^{-1} \mathcal Y q )_{\mathfrak u(n)} &=& -2n \tr (q^{-1} \mathcal X q q^{-1} \mathcal Y q ) = -2n\tr (q^{-1} \mathcal X \mathcal Y q ) \\
&=& -2n\tr (\mathcal Y qq^{-1} \mathcal X ) =-2n\tr(\mathcal X \mathcal Y )= (\mathcal X , \mathcal Y )_{\mathfrak u(n)}
\end{eqnarray*}
for all $\mathcal X , \mathcal Y \in \mathfrak u (n)$ and $q \in U(n)$. 

\begin{rem}
The left and right action of any subgroup $U_n(k)$, $1\leq k\leq n$ on the group $U(n)$ and on the Lie algebra $\mathfrak u(n)$ are defined as a matrix multiplication from the left or from the right. The inner product $( \cdot\,,\cdot )_{\mathfrak g}=-2n\tr(\cdot\,,\cdot)$ on the Lie algebra $\mathfrak u(n)$ is invariant under the adjoint action of $U_n(k)$ and therefore the metric $\langle\cdot\,,\cdot\rangle_{U(n)}$, defined by left or right translations by the action of $U_n(k)$, is bi-invariant under this action.
\end{rem}

%%%%%%%%%%%%%%%%%%%%

\subsection{Stiefel manifold and metric of constant bi-invariant type}

%%%%%%%%%%%%%%%%%%%%

The Stiefel manifold $V_{n,k}$ is the set of all $k$-tuples $(q_1, \dotso , q_k)$ of vectors $q_i \in \mathbb{C}^n$, $ i \in \{1, \dotso , k\}$, which are orthonormal with respect to the standard Hermitian metric. This is a compact manifold which can be equivalently defined as 
$$
V_{n,k}:=\{X\in\mathbb{C}^{n\times k} \vert \ \ \bar{X}^TX=I_{k}\}.
$$
The condition $\bar{X}^TX=I_{k}$ is equivalent to the orthonormality of columns. This $k$ orthonormal columns can be considered as arbitrary $k$ columns in a matrix $X\in U(n)$. This allows us to realize the Stiefel manifold as a quotient set of $U(n)$ by the group $U_n(n-k)$. To do this we introduce the equivalence relation $\backsim_1$ on $U(n)$ by
\begin{eqnarray*}
q \backsim_1 p\quad \Longleftrightarrow\quad q=p \begin{pmatrix} 
I_k & 0 
\\ 0 & U_{n-k}
\end{pmatrix}, \qquad q,p \in U(n),\quad U_{n-k} \in U(n-k). 
\end{eqnarray*}
This results to the
equivalence class for $q\in U(n)$ 
\begin{eqnarray*} 
[q]^{\backsim_1}=\left\{q 
\begin{pmatrix} 
I_k & 0 \\ 0 & U_{n-k} 
\end{pmatrix},
\Big\vert U_{n-k} \in U(n-k) \right\} \in U(n)/U_n(n-k),\ \ q\in U(n).
\end{eqnarray*}
The quotient $U(n)/U_n(n-k)$ is a smooth manifold with the quotient topology and we denote the natural projection from $U(n)$ to the quotient $U(n)/U_n(n-k)$ by $\pi_1$. We identify the equivalence class $[q]^{\backsim_1}$ with a point in the Stiefel manifold and write $[q]_{V_{n,k}}\in V_{n,k}$ instead of $[q]^{\backsim_1}$ to emphasize that point belongs to the Stiefel manifold.
So, practically, an element of $V_{n,k}$ is thought of an element in $U(n)$ whose first $k$ columns from the left are of interest and the last $n-k$ columns are not. The real dimension of $V_{n,k}$ is $2nk-k^2$. 

The tangent space to the Stiefel manifold is a quotient of the tangent space to $U(n)$ by tangent space of the equivalence classes. To obtain it we differentiate curves $\gamma (t) \in [q]^{\backsim_1}$ at $t=0$ for a fixed $q \in U(n)$ and get  the space $\mathcal R=\Big\{q \begin{pmatrix}
0 & 0 \\ 0 & \mathcal C \end{pmatrix} \vert \ \ \mathcal C \in \mathfrak{u}(n-k)\Big\}$. 
Intuitively, movements in the direction $\mathcal R$ make no change in the quotient space. It follows that the tangent space $T_{[q]_{V_{n,k}}}V_{n,k}$ to the Stiefel manifold at $[q]_{V_{n,k}}\in V_{n,k}$ is given by the quotient of the tangent space $T_qU(n)$, that is isomorphic to $q\mathfrak{u}(n)$, by $\mathcal R$:
$$T_{[q]_{V_{n,k}}}V_{n,k}=\left\{ [q]_{V_{n,k}} \begin{pmatrix} \mathcal X_1 & -\bar{\mathcal X_2}^T \\ \mathcal X_2 & 0 \end{pmatrix} 
\Big\vert \ \ \mathcal X_1 \in \mathfrak{u}(k), \mathcal X_2 \in \mathbb{C}^{(n-k) \times k} \right\} . $$
Similar results can be found in \cite{A} or \cite{Ma}.

Now we define a metric $\langle \cdot\, , \cdot \rangle_{V_{n,k}}$ on $V_{n,k}$ by
\begin{eqnarray*}
& \left\langle  [q]_{V_{n,k}} 
\begin{pmatrix} \mathcal X_1 & -\bar{\mathcal X_2}^T \\ \mathcal X_2 & 0 
\end{pmatrix} , 
 [q]_{V_{n,k}} 
 \begin{pmatrix} \mathcal Y_1 & -\bar{\mathcal Y_2}^T \\ \mathcal Y_2 & 0 
 \end{pmatrix} 
 \right\rangle_{V_{n,k}} \Big( [q]_{V_{n,k}} \Big)
 \\
 &:=
 \left\langle q 
 \begin{pmatrix} \mathcal X_1 & -\bar{\mathcal X_2}^T \\ \mathcal X_2 & 0 
 \end{pmatrix} , 
 q \begin{pmatrix} \mathcal Y_1 & -\bar{\mathcal Y_2}^T \\ \mathcal Y_2 & 0 
 \end{pmatrix} 
 \right\rangle_{U(n)} \big(q\big)
 = \left( 
\begin{pmatrix} \mathcal X_1 & -\bar{\mathcal X_2}^T \\ \mathcal X_2 & 0 
\end{pmatrix} ,  
\begin{pmatrix} \mathcal Y_1 & -\bar{\mathcal Y_2}^T \\ \mathcal Y_2 & 0 
\end{pmatrix} 
\right)_{\mathfrak u(n)},
\end{eqnarray*} 
where $q \in [q]_{V_{n,k}}$ is any representative of the equivalence class $[q]_{V_{n,k}}$. It is clear from this definition that the metric $\langle \cdot\, , \cdot \rangle_{V_{n,k}}$ is independent of the choice of the representation.

Since $U_k[q]_{V_{n,k}} =[U_kq]_{V_{n,k}} $ and $[q]_{V_{n,k}} U_k=[qU_k]_{V_{n,k}} $, $U_k\in U_n(k)$, it follows directly from the definition of the metric on $T_{[q]_{V_{n,k}}}V_{n,k}$ and the bi-invariance of the metric $\langle\cdot\,,\cdot\rangle_{U(n)}$ with respect to $U_n(k)$ that 
\begin{eqnarray*}
&&
\left\langle  
[U_k q]_{V_{n,k}} 
\begin{pmatrix} \mathcal X_1 & -\bar{\mathcal X_2}^T \\ \mathcal X_2 & 0 
\end{pmatrix} , [U_k q]_{V_{n,k}} 
\begin{pmatrix} \mathcal Y_1 & -\bar{\mathcal Y_2}^T \\ \mathcal Y_2 & 0 
\end{pmatrix} 
\right\rangle_{V_{n,k}}  
\\ 
&=&
\left( 
\begin{pmatrix} \mathcal X_1 & -\bar{\mathcal X_2}^T \\ \mathcal X_2 & 0 
\end{pmatrix} ,  
\begin{pmatrix} \mathcal Y_1 & -\bar{\mathcal Y_2}^T \\ \mathcal Y_2 & 0 
\end{pmatrix} 
\right)_{\mathfrak u(n)} 
\\
&=& 
\left\langle   
[q]_{V_{n,k}} \begin{pmatrix} \mathcal X_1 & -\bar{\mathcal X_2}^T \\ \mathcal X_2 & 0 
\end{pmatrix} , [q]_{V_{n,k}} 
\begin{pmatrix} \mathcal Y_1 & -\bar{\mathcal Y_2}^T \\ \mathcal Y_2 & 0 
\end{pmatrix} 
\right\rangle_{V_{n,k}}  
\end{eqnarray*}
and
\begin{eqnarray*}
&&
\left\langle 
[qU_k]_{V_{n,k}} 
\begin{pmatrix} \mathcal X_1 & -\bar{\mathcal X_2}^T \\ \mathcal X_2 & 0 
\end{pmatrix}, [qU_k]_{V_{n,k}} 
\begin{pmatrix} \mathcal Y_1 & -\bar{\mathcal Y_2}^T \\ \mathcal Y_2 & 0 
\end{pmatrix}
\right\rangle_{V_{n,k}}  
\\ 
&=&
\left( 
\begin{pmatrix} \mathcal X_1 & -\bar{\mathcal X_2}^T \\ \mathcal X_2 & 0 
\end{pmatrix} ,  
\begin{pmatrix} \mathcal Y_1 & -\bar{\mathcal Y_2}^T \\ \mathcal Y_2 & 0 
\end{pmatrix} 
\right)_{\mathfrak u(n)} 
\\
&=&
\left\langle 
[q]_{V_{n,k}} \begin{pmatrix} \mathcal X_1 & -\bar{\mathcal X_2}^T \\ \mathcal X_2 & 0 
\end{pmatrix} ,  [q]_{V_{n,k}} 
\begin{pmatrix} \mathcal Y_1 & -\bar{\mathcal Y_2}^T \\ \mathcal Y_2 & 0 
\end{pmatrix} 
\right\rangle_{V_{n,k}} ,
\end{eqnarray*}
where $U_k$ is any element in $U_n(k) \subset U(n)$. So the metric of $\langle \cdot\, , \cdot \rangle_{V_{n,k}}$ is invariant under the action of $U_n(k)$.

Now we show that the metric $\langle\cdot\,,\cdot\rangle_{V_{n,k}}$ on $V_{n,k}$ is of constant bi-invariant type with respect to the right group action of $U_n(k)$. 
To prove it we recall that the infinitesimal generator $\sigma_{[q]_{V_{n,k}}}\colon\mathfrak{u}_n(k) \to T_{[q]_{V_{n,k}}}V_{n,k}$ is given by $\sigma_{[q]_{V_{n,k}}}(\xi)=[q]_{V_{n,k}}\xi$, where $\mathfrak{u}_n(k)$ is the Lie algebra of $U_n(k)$. It follows that 
$$
\mathbb{I}_{[q]_{V_{n,k}}}(\xi , \eta)=\langle [q]_{V_{n,k}} \xi , [q]_{V_{n,k}}\eta \rangle_{V_{n,k}}=-2n \tr(\xi
\eta),\quad\text{ where}\quad [q]_{V_{n,k}} \in V_{n,k}.
$$ 
This implies that $\mathbb{I}_{[q]_{V_{n,k}}}(\xi , \eta)$ is independent of $[q]_{V_{n,k}}$. 

%%%%%%%%%%%%%%%%%%%

\subsection{Grassmann manifold}

%%%%%%%%%%%%%%%%%%%

The Grassmann manifold $G_{n,k}$ is defined as a collection of all $k$-dimensional subspaces $\Lambda$ of $\mathbb{C}^n$. Equivalently, an element $\Lambda$ of $G_{n,k}$ can be written as a $(n \times k)$ matrix with columns $e_1, \dotso , e_k$, such that $\spn (e_1, \dotso , e_k)=\Lambda$. We are interested in the representation of $G_{n,k}$ as a quotient of $U(n)$ by some subgroup. As in the previous case of the Stiefel manifold, we quotient $U(n)$ by $U_n(n-k)$, but moreover, since the definition of $G_{n,k}$ does not depend on the choice of the orthonormal basis $e_1, \dotso , e_k$ for $\Lambda$, but only on its span, we also quotient $U(n)$ by the group $U_n(k)$ that leaves $\spn\{e_1, \dotso , e_k\}$ invariant. Therefore, we
define the equivalence relation $\backsim_2$ in $U(n)$ by 
\begin{eqnarray*}
m_1\backsim_2 m_2 \quad \Longleftrightarrow\quad m_1=m_2 \begin{pmatrix} U_k & 0 \\ 0 & U_{n-k} 
\end{pmatrix},\qquad m_1 , m_2 \in U(n),
\end{eqnarray*}
where $U_k \in U(k)$, $U_{n-k} \in U(n-k)$. This leads to the equivalence class 
\begin{eqnarray*}
[m]^{\backsim_2}=\left\{m 
\begin{pmatrix} 
U_k & 0 
\\ 
0 & U_{n-k} 
\end{pmatrix} 
\Big\vert\ \  U_k \in U(k), U_{n-k} \in U(n-k) \right\} \subset U(n),\quad m \in U(n), 
\end{eqnarray*}
which is isomorphic to $U(k) \times U(n-k)\cong U_n(k) \times U_n(n-k) $. We identify $G_{n,k}$ with the quotient space $U(n)/(U_n(k)\times U_n(n-k))$ and use the notation $[m]_{G_{n,k}}$ for $[m]^{\backsim_2}$ in the present Section~\ref{VG}. 

Furthermore, we obtain that the tangent space to the equivalence class $[m]^{\backsim_2}$ is 
$$\left\{m \begin{pmatrix} \mathcal X_1 & 0 \\ 0 & \mathcal X_4 \end{pmatrix} \Big\vert\ \  \mathcal X_1 \in
\mathfrak{u}(k), \ \mathcal X_4 \in \mathfrak{u}(n-k) \right\},\quad m\in U(n),
$$ 
and it implies that the tangent space of $G_{n,k}$ at the point $[m]_{G_{n,k}}$ is given by
$$T_{[m]_{G_{n,k}}}G_{n,k}=\left\{[m]_{G_{n,k}} \begin{pmatrix} 0 & \mathcal X_2 \\ -\bar{\mathcal X_2}^T & 0 \end{pmatrix} \Big\vert \ \ \mathcal X_2 \in \mathbb{C}^{k \times (n-k)} \right\} .$$ 
It has real dimension $2k(n-k)$ that gives the real dimension of $G_{n,k}$, see also~\cite{A,Ma}.

We define a metric $\langle \cdot\, , \cdot\rangle_{G_{n,k}}$ on $G_{n,k}$ by 
\begin{eqnarray*} 
& \left\langle
 [m]_{G_{n,k}} 
\begin{pmatrix} 0 & \mathcal X_2 \\ -\bar{\mathcal X_2}^T & 0 
\end{pmatrix} , [m]_{G_{n,k}} 
\begin{pmatrix} 0 & \mathcal Y_2 \\ -\bar{\mathcal Y_2}^T & 0 
\end{pmatrix} 
\right\rangle_{G_{n,k}} \Big([m]_{G_{n,k}} \Big)
\\
&:=  
\left\langle 
m 
\begin{pmatrix} 0 & \mathcal X_2 \\ -\bar{\mathcal X_2}^T & 0 
\end{pmatrix} , m 
\begin{pmatrix} 0 & \mathcal Y_2 \\ -\bar{\mathcal Y_2}^T & 0 
\end{pmatrix} 
\right\rangle_{U(n)}\big(m\big) 
\\
&= 
\left( 
\begin{pmatrix} 0 & \mathcal X_2 \\ -\bar{\mathcal X_2}^T & 0 
\end{pmatrix} , 
\begin{pmatrix} 0 & \mathcal Y_2 \\ -\bar{\mathcal Y_2}^T & 0 
\end{pmatrix} 
\right)_{\mathfrak u(n)}, 
\end{eqnarray*}
where $m \in U(n)$ is any representative of $[m]_{G_{n,k}}$. 

%%%%%%%%%%%%%%%%%%%%%%%%%

\subsection{Submersion $\pi\colon V_{n,k} \to G_{n,k}
$ and sub-Riemannian geodesics.}\label{sec:submersion}

%%%%%%%%%%%%%%%%%%%%%%%%%

Starting from now, we will consider the matrices $q$ and $m$ as  elements in $U(n)$. Now we can define the submersion 
\begin{eqnarray*}
\pi\colon V_{n,k} &\to& G_{n,k} , \\ {[q]}_{V_{n,k}} &\mapsto& {[m]}_{G_{n,k}}.
\end{eqnarray*} 
The projection $\pi$ sends the equivalence class $[q]^{\backsim_1}$ to the equivalence class $[m]^{\backsim_2}$,
where $m \in U(n)$ can be any matrix from the set 
$$\left\{ q \begin{pmatrix} U_k & 0 \\ 0 & U_{n-k} \end{pmatrix} \Big\vert\ \ U_k \in U(k), U_{n-k} \in U(n-k) \right\}.
$$ 
Note that the latter set consists of all unitary matrices whose first $k$ columns from the left span the same space as the first left $k$ columns of $q$. This implies that a fibre over a point $[m]_{G_{n,k}}\in G_{n,k}$ is given by 
\begin{align*}
\pi^{-1}([m]_{G_{n,k}})
&=\left\{\Big[m \begin{pmatrix} U_k & 0 \\ 0 & I_{n-k} \end{pmatrix} \Big]_{V_{n,k}} \Big\vert \ \ U_k \in U(k) \right\}
\\
&=\left\{[m]_{V_{n,k}}\begin{pmatrix} U_k & 0 \\ 0 & I_{n-k} \end{pmatrix}\Big\vert \ \ U_k \in U(k) \right\},\quad m\in U(n),
\end{align*}
which is homeomorphic to $U_n(k)\cong U(k)$.

The submersion $\pi$ is also a principal $U_n(k)$-bundle, where the right group action is defined by the multiplication from the right by an element from $U_n(k)$. It remains to show that the right action of $U_n(k)$ is continuous, preserves the fibre,
acts freely and transitively on the fibre.  

The multiplication of  $[q]_{V_{n,k}}\in V_{n,k}$ from the right by an element $U_k^0\in U(k)$ is given by
\begin{eqnarray*}
q \begin{pmatrix} I_k & 0 \\ 0 & U_{n-k} \end{pmatrix} \begin{pmatrix} U_k^0 & 0 \\ 0 & I_{n-k} \end{pmatrix} = q \begin{pmatrix} U_k^0 & 0 \\ 0 & U_{n-k}
\end{pmatrix},\quad q\in U(n),
\end{eqnarray*}
where $U_{n-k}$ is an arbitrary element of $U(n-k)$ and $U_k^0$ is a fixed element of $U(k)$. It follows that the right multiplication is well defined and
continuous. It can also be seen, that it preserves the fibre of $\pi^{-1}(\pi([q]_{V_{n,k}}))$. 
By the definition of the fibre it is clear that $[q]_{V_{n,k}}U(k)=\pi^{-1}(\pi([q]_{V_{n,k}}))$ and this implies the transitivity of the $U_n(k)$ action.

To show that $U_n(k)$ acts freely, we assume that $\tilde{U_1}= \begin{pmatrix} U_1 & 0 \\ 0 & I_{n-k} \end{pmatrix}\in U_n(k)$, $\tilde{U_2}=\begin{pmatrix} U_2 & 0 \\ 0 & I_{n-k} \end{pmatrix} \in U_n(k)$
and $[q]_{V_{n,k}}\tilde{U_1}=[q]_{V_{n,k}}\tilde{U_2}$ with $[q]_{V_{n,k}}=\begin{pmatrix} q_1 & q_2 \\ q_3 & q_4 \end{pmatrix}$, $q_1 \in \mathbb{C}^{k \times k}$, $q_2 \in \mathbb{C}^{k \times (n-k)}$, $q_3 \in \mathbb{C}^{(n-k) \times k}$ and $q_4 \in \mathbb{C}^{(n-k) \times (n-k)}$. Then we get the equations
\begin{eqnarray*}
q_1U_1=q_1U_2 &\quad\Longleftrightarrow\quad & q_1=q_1U_2U_1^{-1}=q_1U_1U_2^{-1}, 
\\
q_3U_1=q_3U_2 &\quad\Longleftrightarrow\quad & q_3=q_3U_2U_1^{-1}=q_3U_1U_2^{-1},
 \end{eqnarray*}
which leads to $U_1=U_2$ and so $\tilde{U_1}=\tilde{U_2}$. Thus, we showed that $\pi \colon V_{n,k} \to
G_{n,k}$ is a principal $U_n(k)$-bundle.

The differential of $\pi$ defines the vertical and horizontal spaces. The differential $d_{ [q]_{V_{n,k}}}\pi$ at $ [q]_{V_{n,k}}$ acts as  
$$ [q]_{V_{n,k}}\begin{pmatrix} \mathcal X_1 & \mathcal X_2 \\ -\bar{\mathcal X_2}^T & 0 \end{pmatrix} \mapsto [m]_{G_{n,k}}\begin{pmatrix} 0 & \mathcal X_2 \\ -\bar{\mathcal X_2}^T & 0 \end{pmatrix},$$
where $m$ is defined as above for $\pi$. 
So, the kernel of $d_{ [q]_{V_{n,k}}}\pi$ or the vertical space $\mathcal V_{[q]_{V_{n,k}}}$ is given by 
\begin{eqnarray*}
\mathcal{V}_{ [q]_{V_{n,k}}}=\left\{[q]_{V_{n,k}} \begin{pmatrix} \mathcal X_1 & 0 \\ 0 & 0 \end{pmatrix} \Big\vert\ \ \mathcal X_1 \in \mathfrak{u}(k)\right\} ,\qquad q\in U(n).
\end{eqnarray*}
We choose the horizontal space of $V_{n,k}$ by setting 
\begin{equation}\label{eq:horizontal}
\mathcal{H}_{ [q]_{V_{n,k}}}=\left\{[q]_{V_{n,k}} \begin{pmatrix} 0 & \mathcal X_2 \\ -\bar{\mathcal X_2}^T & 0 \end{pmatrix} \Big\vert \ \ \mathcal X_2 \in \mathbb{C}^{k \times (n-k)} \right\},\qquad q\in U(n).
\end{equation}
It is clear that $d\pi \colon TV_{n,k} \to TG_{n,k}$ is a linear isometry if we restrict it to the horizontal space, $\mathcal{H}_{ [q]_{V_{n,k}}}
\to T_{ \pi([q]_{V_{n,k}})}G_{n,k}$ for each $ [q]_{V_{n,k}} \in V_{n,k}$, therefore $\pi$ is a Riemannian submersion.

The $\mathfrak{u}_n(k)$-valued connection one-form $A_{ [q]_{V_{n,k}}}\colon T_{ [q]_{V_{n,k}}}V_{n,k} \to \mathfrak{u}_n(k)$ is given by 
$$A_ {[q]_{V_{n,k}}}\left(  [q]_{V_{n,k}} \begin{pmatrix} \mathcal X_1 & \mathcal X_2 \\ -\bar{\mathcal X_2}^T & 0 \end{pmatrix} \right):=\begin{pmatrix} \mathcal X_1 & 0 \\ 0 & 0 \end{pmatrix}\in \mathfrak{u}_n(k),\qquad \mathcal X_2 \in \mathbb{C}^{k \times (n-k)}.
$$

Now we can write precisely the normal sub-Riemannian geodesic on $V_{n,k}$ starting from a point $ [q]_{V_{n,k}}$ with initial velocity $v \in T_{ [q]_{V_{n,k}}}V_{n,k}$. It is given by
\begin{eqnarray}\label{sRgeodesic}
\gamma_v (t) &=& \exp_R(tv) \exp_{U_n(k)}(-tA(v)) \nonumber
\\
&=&  \pi_1\left(q\exp_{U(n)}\left(t 
\begin{pmatrix} \mathcal X_1 &  \mathcal X_2 \\ -\bar{ \mathcal X_2}^T & 0  
\end{pmatrix} \right)
\right) 
\exp_{U_n(k)}\left( -t \begin{pmatrix}  \mathcal X_1 & 0 \\ 0 & 0 \end{pmatrix} \right)
\end{eqnarray}
where $q \in U(n)$, $v= [q]_{V_{n,k}} \begin{pmatrix}  \mathcal X_1 &  \mathcal X_2 \\ -\bar{ \mathcal X_2}^T & 0 \end{pmatrix} \in T_{[q]_{V_{n,k}} }V_{n,k}$ with $\begin{pmatrix}  \mathcal X_1 &  \mathcal X_2 \\ -\bar{ \mathcal X_2}^T
& 0 \end{pmatrix} \in \mathfrak{u}(n)$.

We simplify the notation and  from now on write $q\in V_{n,k}$, $m\in G_{n,k}$, $U(k)$ for $U_n(k)$, $U(n-k)$ for $U_n(n-k)$, and $g$ for the Riemannian metric of constant bi-invariant type.

%%%%%%%%%%%%%%%%%%%%%%%%%%
%%%%%%%%%%%%%%%%%%%%%%%%%%

\section{The cut-locus of $V_{n,1}$}\label{sec:Vn1}

%%%%%%%%%%%%%%%%%%%%%%%%%%
%%%%%%%%%%%%%%%%%%%%%%%%%%

In this section we study the cut locus of the Stiefel manifold $V_{n,1}$ considered as a sub-Riemannian manifold by making use of the normal sub-Riemannian geodesics~\eqref{sRgeodesic}.
\begin{defi}
An absolutely continuous horizontal path that realizes the distance between two points is called a minimizing geodesic.
\end{defi}
Recall the definition of the sub-Riemannian cut locus.
\begin{defi}\label{def:cut_locus}
The cut locus of $q_0 \in Q$ in a sub-Riemannian manifold $(Q,\mathcal H,g_{\mathcal H})$  is a set $K_{q_0}\subset Q$ of points reached optimally by more than one horizontal geodesic, i.~e. the cut locus is
\begin{eqnarray*}
K_{q_0}:= \Big\{&q \in Q\ \vert\ \ \text{there exist}\ \  T \in \mathbb{R}^+,\  v_1, v_2 \in T_{q_0}Q,\ v_1 \not = v_2,\ \text{and } 
\\
&\text{minimizing horizontal geodesics} \ \ \gamma _{v_1}(t),\ \gamma _{v_2}(t),\ \ \text{starting from}\ \ q_0,\ \ \text{and}
\\
 & \gamma _{v_1}(T)= \gamma _{v_2}(T) =q\Big\} .
\end{eqnarray*}
If we replace minimizing horizontal geodesics into minimizing normal horizontal geodesics we obtain a definition of the normal sub-Riemannian cut locus. Further on we will work with cut locus, given in Definition~\ref{def:cut_locus}.
\end{defi}
Starting from now we will write $Id$ for the equivalence class $[I_n]_{V_{n,k}} \in V_{n,k}$. The main theorem is stated as following.

\begin{theorem}\label{th:Vn1}
The cut locus $K_{Id}$ on $V_{n,1}$ is given by 
$$L:=\left\{\left[\begin{pmatrix} C & 0 \\ 0 & D \end{pmatrix}\right]_{V_{n,1}} \Big\vert \ C \in U(1),\ D \in U(n-1) \right\} \setminus \left\{Id\right\}.
$$
\end{theorem}

Before we present the proof of Theorem~\ref{th:Vn1} we consider in details the particular case for $n=2$, $k=1$. It allows to understand the general idea of the proof without using tough technical calculations. 

%%%%%%%%%%%%%%%%%%%%%%%%%%%%%%%%%%%%%%

\subsection{The cut locus of $V_{2,1}$}\label{subsec:V21}

%%%%%%%%%%%%%%%%%%%%%%%%%%%%%%%%%%%%%%

Observe that $V_{2,1}$ is three dimensional and the distribution~\eqref{eq:horizontal} is strongly bracket generating. Recall the definition.

\begin{defi}\label{strbrack}
A smooth distribution $\mathcal H$ on $M$ is strongly bracket generating if for any non-zero section $\mathcal Z$ of $\mathcal H$, the tangent bundle $TM$ is generated by $\mathcal H$ and $[\mathcal Z,\mathcal H]$.
\end{defi}

For the manifold $V_{2,1}$ Definition~\ref{strbrack} is reduced to the statement that there exist two sections $\mathcal Z_1$ and $\mathcal Z_2$ of $\mathcal H$ such that $\spn \{\mathcal Z _1(q),\mathcal Z_2(q), [\mathcal Z_1,\mathcal Z_2](q)\}=T_qV_{2,1}$ for all 
$q\in V_{2,1}$.  We can choose, $\mathcal Z_1(q):=q\begin{pmatrix} 0 & 1 \\ -1 & 0 \end{pmatrix} $ and
$\mathcal Z_2(q):=q \begin{pmatrix} 0 & i \\ i & 0 \end{pmatrix}$. It is known, see for instance~\cite{B,M}, that on sub-Riemannian manifolds with strongly bracket generating distributions all minimizing geodesics are normal.

The tangent spaces at
the identity are given by
$$ T_{\Id}V_{2,1}=\left\{ \Id 
\begin{pmatrix} x_1 & x_2 \\ -\bar{x}_2 & 0 
\end{pmatrix} 
\Big\vert\ x_1=\lambda i,\ \lambda \in  \mathbb{R},\ x_2 \in \mathbb{C} \right\}
$$
and
$$ 
T_{\Id}Gr_{2,1}=\left\{ \Id 
\begin{pmatrix} 0 & x_2 \\ -\bar{x}_2 & 0 
\end{pmatrix} \Big\vert\ x_2 \in \mathbb{C} 
\right\}.$$
For a given initial vector $v=\Id \begin{pmatrix} x_1 & x_2 \\ -\bar{x}_2 & 0 \end{pmatrix} \in T_{\Id}V_{2,1}$ a normal sub-Riemannian
geodesic is written as
\begin{eqnarray*}
\gamma _{v}(t) 
&=& 
\pi_1(\exp_{U(2)} (tv)) \exp_{U(1)} \left(-t \begin{pmatrix} \lambda i & 0 \\ 0 & 0 \end{pmatrix} \right) 
\\
&=& 
\pi_1\left(\exp_{U(2)}\left( t \begin{pmatrix} \lambda i & x_2 \\ -\bar{x}_2 & 0 
\end{pmatrix} \right)\right) 
\begin{pmatrix} e^{- \lambda t i} & 0 \\ 0 & 1 \end{pmatrix} 
\\
&=& 
\pi_1\left(\begin{pmatrix} \gamma_v^1(t) & \gamma_v^2(t) \\ \gamma_v^3(t) & \gamma_v^4(t) \end{pmatrix}\right) 
=
 \left[\begin{pmatrix} \gamma_v^1(t) & \gamma_v^2(t) \\ \gamma_v^3(t) & \gamma_v^4(t)  \end{pmatrix}
\right]_{V_{n,k}}
\end{eqnarray*}
with
\begin{eqnarray*}
\gamma_v^1(t)
&=&
\Big(\frac {\lambda}{2\sqrt{\lambda ^2 +4x_2 \bar{x}_2 } }+ \frac {1}{2}\Big) \mu _1(- \lambda,x_2 , t) 
+ \Big(-\frac {\lambda}{2 \sqrt{\lambda ^2 +4x_2
\bar{x}_2 } }+ \frac {1}{2}\Big) \mu _2(- \lambda,x_2 , t), 
\\
\gamma_v^2(t)
&=& 
\frac {x_2 i}{\sqrt{\lambda ^2 +4x_2 \bar{x}_2 }}
\big(\mu _2 ( \lambda,x_2 , t) -\mu _1 ( \lambda,x_2 , t) \big), 
\\
\gamma_v^3(t)
&=& 
\frac{i}{4 x_2} 
\left( \frac{ \lambda ^2 }{ \sqrt{\lambda ^2 +4x_2 \bar{x}_2 }} - \sqrt{\lambda ^2 +4x_2 \bar{x}_2 } \right) 
\big( \mu _2(-\lambda , x_2 ,
t) - \mu _1(-\lambda , x_2 , t)\big)  
\\
&=& 
- \frac{\bar{x}_2i}{\sqrt{\lambda ^2 +4x_2 \bar{x}_2 }}
\big( \mu _2(-\lambda , x_2 , t) - \mu _1(-\lambda , x_2 , t)\big), 
\\
\gamma_v^4(t)
&=& 
- \frac{\mu_1( \lambda,x_2 , t)}{2 \sqrt{\lambda ^2 +4x_2 \bar{x}_2 }} 
\big(\lambda - \sqrt{\lambda ^2 +4x_2 \bar{x}_2 }\big) 
+ \frac{\mu_2(
\lambda,x_2 , t)}{2 \sqrt{\lambda ^2 +4x_2 \bar{x}_2 }} 
\big(\lambda + \sqrt{\lambda ^2 +4x_2 \bar{x}_2 }\big),
\end{eqnarray*}
where 
$$\mu_1(\lambda, x_2, t)= e^{\frac{ti}{2} ( \lambda + \sqrt{\lambda ^2 +4x_2 \bar{x}_2 })} \ \text{and}\ 
\mu_2(\lambda ,x_2, t)=e^{\frac{ti}{2} ( \lambda
- \sqrt{\lambda ^2 +4x_2 \bar{x}_2 })}.
$$
In calculations we used the diagonalization of the matrix $t \begin{pmatrix} \lambda i & x_2 \\ -\bar{x}_2 & 0 \end{pmatrix}=SDS^{-1}$ with 
\begin{equation*}
S = 
\begin{pmatrix} 1 & 1 \\  \\- \frac{i}{2x_2} (\lambda - \sqrt{\lambda ^2 +4x_2 \bar{x}_2 }) & - \frac{i}{2x_2} (\lambda + \sqrt{\lambda ^2 +4x_2
\bar{x}_2 }) 
\end{pmatrix}. 
\end{equation*}
\begin{equation*}
S^{-1}
= \begin{pmatrix} \frac{\lambda}{2\sqrt{\lambda ^2 +4x_2 \bar{x}_2 }}+\frac{1}{2} & & - \frac{x_2i}{\sqrt{\lambda ^2 +4x_2 \bar{x}_2 }} 
\\ \\ \frac{1}{2}
- \frac{\lambda}{2\sqrt{\lambda ^2 +4x_2 \bar{x}_2 }} &  &\frac{x_2 i}{\sqrt{\lambda ^2 +4x_2 \bar{x}_2 }} 
\end{pmatrix}, 
\end{equation*}
\begin{equation*}
D= \begin{pmatrix} it(\frac{\lambda + \sqrt{\lambda ^2 +4x_2 \bar{x}_2 }}{2}) & 0 \\ \\  0 & it( \frac{\lambda - \sqrt{\lambda ^2 +4x_2 \bar{x}_2 }}{2})
\end{pmatrix},
\end{equation*}
in order to express $\exp_{U(2)}\left( t \begin{pmatrix} \lambda i & x_2 \\ -\bar{x}_2 & 0 \end{pmatrix} \right)=S \exp_{U(2)}(D) S^{-1}$. 

\begin{lemma}\label{cutloci12}
The set 
$$L:=\left\{\left[\exp_{U(2)} \begin{pmatrix} c_1i & 0 \\ 0 & c_2i \end{pmatrix} \right]_{V_{2,1}} \Big\vert\ c_1, c_2 \in \mathbb{R} \right\} \setminus \{\Id\}
$$ is the cut locus $K_{\Id}$ of
$V_{2,1}$.
\end{lemma}

\begin{proof}
It is clear that it is enough to concentrate on the calculation of the first column $\begin{pmatrix} \gamma_v^1(t)  \\  \gamma_v^3(t)\end{pmatrix}$ in the equivalence class $\left[\begin{pmatrix} \gamma_v^1(t) & \gamma_v^2(t) \\ \gamma_v^3(t) & \gamma_v^4(t) \end{pmatrix}\right]_{V_{2,1}}$. We show first that if $q\in L$, then there are several minimizing geodesics reaching $q$ in the same time.

Suppose there exists an initial vector $v^*= \begin{pmatrix} \lambda^* i & x_2^* \\ -\bar{x}_2^* & 0 \end{pmatrix} $ with $x_2^*\not =0$, and $T \in \mathbb{R}^+$ such that the minimizing geodesic $\gamma
_{v^*}$ connects $\Id\in V_{2,1}$ with 
$$q=\gamma
_{v^*}(T^*)
= \left[\begin{pmatrix} e^{c_1i} & 0 \\ 0 & e^{c_2i} \end{pmatrix}\right]_{V_{2,1}} \in L.
$$ 
We see that $\gamma_{v^*}^2(T^*)=0$. It implies the following equivalences
\begin{eqnarray}\label{equivalences}
&  \qquad& \mu _1(\lambda^*, x_2^*, T^*)=\mu _2(\lambda^*, x_2^*, T^*) 
\\
& \Longleftrightarrow & e^{\frac{T^*i}{2} ( \lambda^*+ \sqrt{(\lambda^*) ^2 +4|x_2|^2 })} = e^{\frac{T^*i}{2} ( \lambda^* - \sqrt{(\lambda^*)
^2 +4|x_2|^2})} \nonumber
\\
& \Longleftrightarrow & e^{\frac{T^*i}{2} \sqrt{(\lambda^*)^2 +4|x_2|^2}}= e^{- \frac{T^*i}{2} \sqrt{(\lambda^*)^2 +4|x_2|^2} }\nonumber
\\
& \Longleftrightarrow & \frac{T^*}{2} \sqrt{(\lambda^*) ^2 +4x_2^* \bar{x}_2^*} = k \pi, \quad\text{for some}\quad k\in\mathbb Z.\nonumber
\end{eqnarray}
Let us fix such $k\in\mathbb Z$ and note
\begin{eqnarray*}
\mu _1(\lambda^*, x_2^*, T^*) 
= 
e^{\frac{T^*i}{2} ( \lambda^*+ \sqrt{(\lambda^*) ^2 +4|x_2|^2})}
= \pm e^{\frac{T^*i}{2}  \lambda^*} 
=\mu _2(\lambda^*, x_2^*, T^*).
\end{eqnarray*}
We conclude that functions $\mu _{1}(\lambda^*, x_2^*, T^*)$ and $\mu _{2}(\lambda^*, x_2^*, T^*)$ are independent of $x_2^*$ itself, but depend on the norm $|x_2^*|^2=x_2^*\bar x_2^*$. Let us pick up another initial velocity vector
$v_1 = \begin{pmatrix} \lambda^* i & y_2 \\ -\bar{y}_2 & 0 \end{pmatrix} $ with $|y_2|^2=|x_2^*|^2$ and $x_2^*\not =y_2$. Then $\gamma _{v_1}(T^*)=q$. 

In the next step we show that the length of the geodesic $\gamma _{v_1}$ coincides with the length of the minimizing geodesic $\gamma _{v^*}$. We actually claim that the length of any geodesic $\gamma _{v}$ with $v= \begin{pmatrix} \lambda i & x_2 \\ -\bar{x}_2 & 0 \end{pmatrix} $ depends on the fixed final time $T$ and the norm $|x_2|$. 

We recall that the square of the length of the velocity vector $\dot{\gamma}_{v}(t)$ is given by
$$\langle\dot{\gamma}_{v}(t),\dot{\gamma}_{v}(t)\rangle_{V_{2,1}}=-2n\tr\big([\gamma_v(t)^{-1}\dot{\gamma}_{v}(t)]^2\big).$$
Fix a point
$p(t)=\exp_{U(2)}\left( t \begin{pmatrix} \lambda i & x_2 \\ -\bar{x}_2 & 0 \end{pmatrix} \right)
\begin{pmatrix} e^{- \lambda t i} & 0 \\ 0 & 1 \end{pmatrix}\in U(2)$ such that $\gamma_v(t)=\pi_1\big(p(t)\big)$. 
To calculate 
$\dot{\gamma}_{v}(t)=d_{p(t)}\pi_1p'(t)$ 
we use the chain rule 
\begin{eqnarray*}
\dot{\gamma}_{v}(t) &=& d_{p(t)}\pi_1\bigg[ \exp_{U(2)} \left\{  t \begin{pmatrix} \lambda i & x_2 \\ -\bar{x}_2 & 0 \end{pmatrix} \right\} \begin{pmatrix} \lambda i & x_2 \\ -\bar{x}_2 & 0
\end{pmatrix} \begin{pmatrix} e^{- \lambda t i} & 0 \\ 0 & 1 \end{pmatrix} 
\\ 
&+& \exp_{U(2)}\left\{ t \begin{pmatrix} \lambda i & x_2 \\ -\bar{x}_2 & 0 \end{pmatrix} \right\} 
\begin{pmatrix} e^{- \lambda t i} & 0 \\ 0 & 1 \end{pmatrix}
\begin{pmatrix} - \lambda i & 0 \\ 0 & 0 \end{pmatrix} \bigg]
\\
&=& d_{p(t)}\pi_1\bigg[\exp_{U(2)}\left\{ t \begin{pmatrix} \lambda i & x_2 \\ -\bar{x}_2 & 0 \end{pmatrix} \right\} \begin{pmatrix} e^{- \lambda t i} & 0 \\ 0 & 1 \end{pmatrix}
\begin{pmatrix} \lambda i & x_2 e^{\lambda i t} \\ - \bar{x}_2 e^{- \lambda i t} & 0 \end{pmatrix} 
\\
&+& 
\exp_{U(2)}\left\{ t \begin{pmatrix} \lambda i & x_2 \\ -\bar{x}_2 & 0 \end{pmatrix} \right\} 
\begin{pmatrix} e^{- \lambda t i} & 0 \\ 0 & 1 \end{pmatrix}
\begin{pmatrix} - \lambda t i & 0 \\ 0 & 0 \end{pmatrix}\bigg] 
\\
&=& 
d_{p(t)}\pi_1\left[\exp_{U(2)}\left\{ t \begin{pmatrix} \lambda i & x_2 \\ -\bar{x}_2 & 0 \end{pmatrix} \right\} 
\begin{pmatrix} e^{- \lambda t i} & 0 \\ 0 & 1 \end{pmatrix}
\begin{pmatrix} 0 & x_2 e^{\lambda i t} \\ - \bar{x}_2 e^{- \lambda i t} & 0 \end{pmatrix}\right] 
\\
&=& 
\gamma_{v}(t) \begin{pmatrix} 0 & x_2 e^{\lambda i t} \\ - \bar{x}_2 e^{- \lambda i t} & 0 \end{pmatrix} .
\end{eqnarray*}
It follows that 
$
\langle\dot{\gamma}_{v}(t),\dot{\gamma}_{v}(t)\rangle_{V_{2,1}}= -4 \tr 
\begin{pmatrix} -|x_2|^2 & 0 \\ 0 & -|x_2|^2 \end{pmatrix} =8|x_2|^2
$.
Since the length of the geodesic $\gamma _{v}$ depends only on $T$ and the norm $|x_2|$ we conclude that 
$\gamma_{v_1}$ is a minimizing geodesic from the identity to $q$. With this we finished to show the inclusion~$L\subset K_{\Id}$. 

To  prove the converse inclusion $K_{\Id} \subset L$ we use a contradiction.  
Suppose $q \in V_{2,1} \setminus L$, but $q\in K_{\Id}$, i.~e. there exist $v_1 = \begin{pmatrix} \lambda_1 i & x_2 \\ -\bar{x}_2 & 0 \end{pmatrix}$ , $v_2 = \begin{pmatrix}
\lambda_2 i & y_2 \\ -\bar{y}_2 & 0 \end{pmatrix} \in \mathfrak{u}(n)$ with $v_1 \not = v_2$ such that $\gamma _{v_1}$ and $\gamma_{v_2}$ are
minimizing geodesics from the identity to $q$, which reach the point $q$ for the first time at the moment $T \in \mathbb{R}^{+}$. Note that values $x_2$ and $y_2$ do not vanish as
otherwise $\gamma_{v_{1}}(t)=\gamma_{v_{2}}(t)=\Id$ for all $t \in \mathbb{R}$. 

We observe that for any unitary matrix $q=\begin{pmatrix} q_1 & q_2 \\ q_3 & q_4 \end{pmatrix}$ one obtains $q_2\neq 0\ \ \Leftrightarrow\ \ q_3\neq 0$.
It follows that if $q\in V_{2,1} \setminus L$, then
\begin{equation*}
\gamma_{v_1}^2\neq 0,\ \gamma_{v_1}^3 \neq0 \ \ \Longleftrightarrow \ \  \mu _2 ( \lambda,x_2 , T) \neq \mu _1 ( \lambda,x_2 , T)
 \ \ \Longleftrightarrow \ \ \frac{T}{2}\sqrt{\lambda ^2 +4|x_2|^2} \not \in \pi \mathbb{Z}
\end{equation*}
by~\eqref{equivalences}. It immediately implies 
$T< \min\Big\{\frac{2 \pi}{\sqrt{\lambda_1^2+4|x_2|^2}} , \frac{2 \pi}{\sqrt{\lambda_2^2+4|y_2|^2} }\Big\}$. 
In the next step we show that neither of these inequalities can be realized under our assumptions.

{\textbf{Case 1.} } Assume that $|x_2| \neq |y_2|$ and $\lambda_1$, $\lambda_2$ are arbitrary.
Then $g(v_1,v_1)=8|x_2|^2 \not = 8|y_2|^2=g(v_2,v_2)$, that implies that the length of both minimizing geodesics $\gamma_{v_1}$ and $\gamma_{v_2}$ is different, which is a
contradiction to the assumption that they are both minimizing at the same time.

{\textbf{Case 2}.} Let $|x_2|=|y_2|$ and $\lambda_1 = \lambda_2$.
As $\gamma_{v_1}^3(T) =\gamma_{v_2}^3(T) \neq 0$ it follows that
\begin{align*}
\gamma_{v_1}^3(T) = &-\frac {\bar{x}_2 i}{\sqrt{\lambda_1 ^2 +4|x_2|^2  }}(\mu _2 ( \lambda_1,x_2 , T) -\mu _1 ( \lambda_1,x_2 , T) ) 
\\
= &-\frac {\bar{y}_2
i}{\sqrt{\lambda_2 ^2 +4|y_2|^2 }}(\mu _2 ( \lambda_2,y_2 , T) -\mu _1 ( \lambda_2,y_2 , T) ) =\gamma_{v_2}^3(T) 
\\
\text{if and only if}&\qquad x_2=y_2,
\end{align*}
because $\sqrt{\lambda_1 ^2 +4|x_2|^2}=\sqrt{\lambda_2 ^2 +4|y_2|^2}$,  $\mu _{1} ( \lambda_1,x_2 , T)=\mu _{1} ( \lambda_2,y_2 , T)$, and $\mu _{2} ( \lambda_1,x_2 , T)=\mu _{2} ( \lambda_2,y_2 , T)$ by definition.
But the equality $x_2=y_2$ implies $v_1=v_2$, that leads to a contradiction.

{\textbf{Case 3}.} Finaly we suppose that $\lambda_1 \neq \lambda_2$ and $|x_2|=|y_2|$.
As in the previous case the equality $\gamma^3_{v_1}(T)=\gamma^3_{v_2}(T)$ implies $| \gamma^3_{v_1}(T) | =|\gamma^3_{v_2}(T) |$ and
$$
\frac{1}{ \sqrt{\lambda_1 ^2 +4|x_2|^2 }} | \mu_2(\lambda_1 , x_2, T)-\mu_1(\lambda_1 , x_2, T) | = \frac{1}{ \sqrt{\lambda_2 ^2 +4|y_2|^2}} |\mu_2(\lambda_2 , y_2, T) -\mu_1(\lambda_2 , y_2, T)|.
$$
Taking into account $|\exp \Big(\frac{Ti\lambda_{j}}{2}\Big)|=1$ for $j=1,2$, we obtain
\begin{eqnarray*}
|\mu_2(\lambda_j , x_2, T)-\mu_1(\lambda_j , x_2, T) |
&=& | e^{-\frac{Ti}{2} \sqrt{ \lambda_j^2+4|x_2|^2}}-e^{\frac{Ti}{2} \sqrt{
\lambda_j^2+4|x_2|^2}} |
 \\
&=& 2 \sin \Big(\frac{T}{2} \sqrt{ \lambda_j^2+4|x_2|^2}\Big),\quad j=1,2,
\end{eqnarray*}
as $\sin x>0$ for $x \in (0,\pi)$. These both equations lead to 
\begin{eqnarray*}
\frac{2 \sin \big(\frac{T}{2} \sqrt{ \lambda_1^2+4|x_2|^2}\big)} {\sqrt{ \lambda_1^2+4|x_2|^2}} 
&=& \frac{2 \sin (\frac{T}{2} \sqrt{
\lambda_2^2+4|y_2|^2})} {\sqrt{ \lambda_2^2+4|y_2|^2}} 
\\
\Longleftrightarrow\quad \frac{2 \sin (\frac{T}{2} \sqrt{ \lambda_1^2+4|x_2|^2})} {T\sqrt{ \lambda_1^2+4|x_2|^2}} 
&=& \frac{2 \sin (\frac{T}{2} \sqrt{
\lambda_2^2+4|y_2|^2})} {T\sqrt{ \lambda_2^2+4|y_2|^2}}.
\end{eqnarray*}
Since the function $\frac{\sin x}{x}$ is injective on the interval $(0,\pi)$ we obtain $\frac{T}{2} \sqrt{ \lambda_1^2+4|x_2|^2}=\frac{T}{2} \sqrt{ \lambda_2^2+4
|y_2|^2}$ which is equivalent to $\lambda_1=\pm \lambda_2$. 

We only need to consider the case $\lambda_1=- \lambda_2$. Note that $\mu_{j} (\lambda_2 , x_2,T)=\mu_{j}(\lambda_2 , y_2,T)$, $j=1,2$, and
$$
\mu_{1} (-\lambda_2 , x_2,T)=\frac{1}{\mu_{2}(\lambda_2 , y_2,T)},\qquad
\mu_{2} (-\lambda_2 , x_2,T)=\frac{1}{\mu_{1}(\lambda_2 , y_2,T)}.
$$ From this it follows that
\begin{equation}\label{tan}
\gamma^1_{v_1}(T)=\gamma^1_{v_2}(T)\quad
\Longleftrightarrow\quad \frac{\tan (\frac{T}{2} \sqrt{ \lambda_2^2+4|y_2|^2})}{\sqrt{ \lambda_2^2+4|y_2|^2} }  =  \frac{\tan (\frac{T \lambda_2}{2})}{\lambda_2}.
\end{equation} 
Since $0<\frac{T \lambda_2}{2}<\frac{T}{2} \sqrt{ \lambda_2^2+4y_2\bar{y}_2}< \pi$ the equality~\eqref{tan} is not true.

\begin{figure}[h]\label{tanxx}
 \centering
 \includegraphics{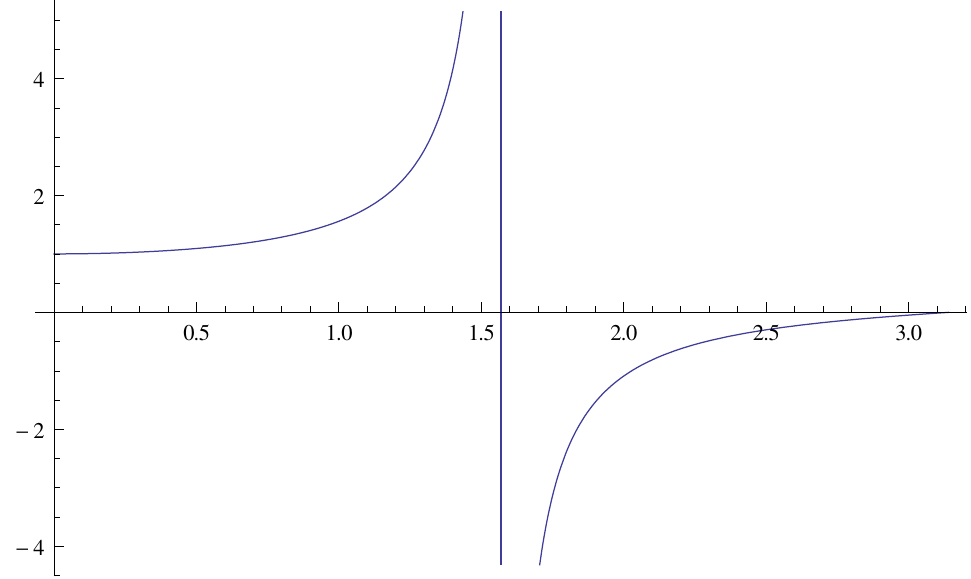}
 % pictures.pdf: 259x156 pixel, 72dpi, 9.14x5.50 cm, bb=0 0 259 156
 \caption{$\frac{\tan (x) }{x}$ on the interval $[0 \,, \pi ]$}
 \label{fig:pic1}
\end{figure}

Figure~\ref{tanxx} illustrates that $y_1<y_2$ implies $\frac{\tan y_1}{y_1}\neq \frac{\tan y_2}{y_2}$. 
The similar calculations can be found in~\cite[p.~1871]{B}.

These 3 cases finish the proof of the theorem.
\end{proof}

%%%%%%%%%%%%%%%%

\subsection{Isomorphism between $V_{2,1}$ and $SU(2)$}

%%%%%%%%%%%%%%%%

In this subsection we show that the results obtained above recover the results obtained in~\cite{B}. 
An element $q$ of $V_{2,1}$ is an equivalence class which can be written as
$$ [q]_{V_{2,1}}= \left\{ \begin{pmatrix} \alpha & \exp ( \lambda i ) \bar{\beta} \\ \beta & -\exp (\lambda  i) \bar{\alpha} \end{pmatrix} \Big\vert  \lambda \in (0, 2 \pi ) \right\}.
$$
Since $\begin{pmatrix} \alpha & \exp ( \lambda i ) \bar{\beta} \\ \beta & -\exp (\lambda i) \bar{\alpha} \end{pmatrix}$ is a unitary matrix we know
that the norm $\lVert \alpha \rVert ^2 +
\lVert \beta \rVert ^2 $ of the vector $\begin{pmatrix} \alpha \\ \beta \end{pmatrix}$ is equal one. Furthermore, we note that points $q\in V_{2,1}$ can be parametrized by the vector $\begin{pmatrix} \alpha \\ \beta \end{pmatrix}$.
Recall the definition of the group $SU(2)=\left\{ \begin{pmatrix} \alpha & \beta \\ -\bar{\beta} & \bar{\alpha} \end{pmatrix} \Big\vert \lVert \alpha \rVert ^2 + \lVert \beta
\rVert ^2 =1 \right\}$. So, it is clear that every element of $SU(2)$ can be represented by the vector $\begin{pmatrix} \alpha \\ \beta \end{pmatrix}$. It follows
that the both manifolds are diffeomorphic under the mapping $f\colon V_{2,1} \to SU(2)$, $[g]_{V_{2,1}} \mapsto
\begin{pmatrix} \alpha & \beta \\ -\bar{\beta} & \bar{\alpha} \end{pmatrix}$. The metric in both cases is left invariant, arising from an inner product on Lie algebras making basis of Lie algebra orthogonal. The horizontal distribution is orthogonal to the vertical one.

\begin{remark}
The set $L$ as a subset of $V_{2,1}$ depends only on $c_1 \in (0,2\pi)$, since the part depending on $c_2$ is quotient out. This implies that
the cut locus of $SU(2)$, which is given by $\left\{ \begin{pmatrix} \exp (c_1 i ) & 0 \\ 0 & \exp (-c_1 i ) \end{pmatrix} \Big\vert c_1 \in (0,2\pi ) \right\}$~\cite{B}, has a bijective relation under the map $f$ to the cut locus of $V_{2,1}$,  given in Lemma~\ref{cutloci12}. 
\end{remark}

Now we proceed to the proof of Theorem~\ref{th:Vn1} that describes the cut locus from the identity on $V_{n,1}$.

\begin{proof} We only need to show the inclusion $K_{Id}\subset L$ since the converse inclusion $L\subset K_{Id}$ will be proved in 
Theorem~\ref{cutlocikn} for the more general case $V_{n,k}$.  

First of all we claim that in the case of $V_{n,1}$ there are no abnormal minimizing geodesics because the distribution is strongly bracket generating. To show that the horizontal distribution is strongly bracket generating we consider an arbitrary element $\begin{pmatrix} 0 & B \\ -\bar{B}^T & 0 \end{pmatrix} \in \mathcal{H}_{\Id}$ with $B=(b_1, \dotso , b_{n-1}) \in \mathbb{C}^{1 \times (n-1)}$. Take basis elements $\begin{pmatrix} 0 & E_{mj} \\ -\bar{E}_{mj}^T & 0 \end{pmatrix}$, $m\in\{0,1\}$, and $j\in\{1, \dotso , n-1\}$. Here  
$E_{mj}\in \mathbb{C}^{1 \times (n-1)}$ is the row with entry $i^m$ at the place $j$ and zeros everywhere else.
Then, the commutator is written as
\begin{eqnarray*}
\left[\begin{pmatrix} 0 & B \\ -\bar{B}^T & 0 \end{pmatrix},\begin{pmatrix} 0 & E_{mj} \\ -\bar{E}_{mj}^T & 0 \end{pmatrix}\right] 
&=& 
\begin{pmatrix} -B\bar{E}_{mj}^T + E_{mj} \bar{B}^T & 0 \\ 0 & 0 \end{pmatrix},
\end{eqnarray*}
 where 
$$-B\bar{E}_{mj}^T + E_{mj}\bar{B}^T = \begin{cases} -2i\im(b_j), & m=0 \\ -2i\re(b_j), & m=1 \end{cases}.$$
An arbitrary choice of $B$ and linearity of Lie bracket imply that any $B$ generates the whole vertical space which allows to conclude that the distribution $\mathcal H$ is strongly bracket generating. 

Now we calculate the precise form of a geodesic $\gamma_v$, concentrating on components $\gamma_v^1$ and $\gamma_v^3$. 
Let the velocity vector be given by $v=\begin{pmatrix} xi & B \\ -\bar{B}^T & 0 \end{pmatrix}$, where 
$x \in \mathbb{R}$ and $B\in \mathbb{C}^{1 \times (n-1)}$.  Recall $\exp(tv)=\sum_{n=0}^{\infty}{\frac{t^n}{n!}v^n}$. 
First we will calculate the two parts of $v^n:=v(n)=\begin{pmatrix} v_1(n) & v_2(n) \\ v_3(n) & v_4(n) \end{pmatrix}$, namely $v_1(n)$ and $v_2(n)$. From the recursion formula $v^n=v^{n-1}v$ it follows that 
$$v_1(n)=v_1(n-1)xi-v_2(n-1)\bar{B}^T=v_1(n-1)xi-v_1(n-2)B\bar{B}^T$$ as $v_2(n)=v_1(n-1)B$. Furthermore, as $v^n=vv^{n-1}$ 
we deduce $$v_3(n)=-\bar{B}^Tv_1(n-1).$$
Having the initial values $v_1(0)=1$, $v_1(1)=xi$, and $v_3(0)=0$ we obtain that 
\begin{eqnarray*}
v_1(n)
&=&
\frac{2^{-n-1}}{i\sqrt{x^2+4B\bar{B}^T}}\big(ix((i\sqrt{x^2+4B\bar{B}^T}+ix)^n-(ix-i\sqrt{x^2+4B\bar{B}^T})^n\big) 
\\
&+&
i\sqrt{x^2+4B\bar{B}^T}\big((ix-i\sqrt{x^2+4B\bar{B}^T})^n+(i\sqrt{x^2+4B\bar{B}^T}+ix)^n)\big) 
\end{eqnarray*}
and we obtain for $\exp(tv):=\begin{pmatrix} \exp(tv)_1 & \exp(tv)_2 \\ \exp(tv)_3 & \exp(tv)_4 \end{pmatrix}$
\begin{eqnarray*}
\exp(tv)_1
&=&
\sum_{n=0}^{\infty}{\frac{t^n}{n!}v_1(n)} 
=\frac{1}{2i\sqrt{x^2+4B\bar{B}^T}}\left(e^{-\frac{it}{2}(\sqrt{x^2+4B\bar{B}^T}-x}\right) 
\\ 
&\times & 
\Big(i\sqrt{x^2+4B\bar{B}^T} \big(e^{it\sqrt{x^2+4B\bar{B}^T}}+1\big)+ix\big(e^{it\sqrt{x^2+4B\bar{B}^T}}-1)\big)\Big) 
\\
&=& 
\frac{1}{2\sqrt{x^2+4B\bar{B}^T}}\left(e^{-\frac{it}{2}(\sqrt{x^2+4B\bar{B}^T}-x}\right) 
\\ 
&\times&
\Big(\sqrt{x^2+4B\bar{B}^T}\big(e^{it\sqrt{x^2+4B\bar{B}^T}}+1\big)+x\big(e^{it\sqrt{x^2+4B\bar{B}^T}}-1)\big)\Big).
\end{eqnarray*}
The first part $\gamma^1_v(t)$ of the normal geodesic $\gamma_v(t)=\begin{pmatrix} \gamma^1_v(t) &\gamma^2_v(t) \\ \gamma^3_v(t) & \gamma^4_v(t)\end{pmatrix}$ is written as
\begin{eqnarray}\label{eq:g1}
\gamma_v^1(t)
&=& 
\exp_{U(n)}(tv)_1\exp_{U(1)}(-tix ) 
=\frac{1}{2\sqrt{x^2+4B\bar{B}^T}}e^{-\frac{it}{2}(\sqrt{x^2+4B\bar{B}^T}+x)}\nonumber
\\
&\times& \left(\sqrt{x^2+4B\bar{B}^T}
(e^{it\sqrt{x^2+4B\bar{B}^T}}+1\big)\right.
+
\left.x(e^{it\sqrt{x^2+4B\bar{B}^T}}-1)\right),
\end{eqnarray}
that coincides with calculations in the case $n=2$. 

The second important part of the geodesic is
\begin{eqnarray*}
\exp(tv)_3
&=& 
\sum_{n=0}^{\infty}{\frac{t^n}{n!}v_3(n)}=\sum_{n=1}^{\infty}{\frac{t^n}{n!}v_3(n)} 
\\
&=& 
\sum_{n=0}^{\infty}{\frac{t^{n+1}}{(n+1)!}v_3(n+1)}=\sum_{n=0}^{\infty}{\frac{t^{n+1}}{(n+1)!}\big(-\bar{B}^Tv^1(n)} \big)
\\
&=& 
-\bar{B}^T\frac{1}{i\sqrt{x^2+4B\bar{B}^T}}e^{-\frac{ti}{2}(\sqrt{x^2+4B\bar{B}^T}-x)}
\left(e^{ti\sqrt{x^2+4B\bar{B}^T}}-1\right),
\end{eqnarray*}
\begin{eqnarray}\label{eq:g2}
\gamma_v^3(t)
&=&
\exp_{U(n)}(tv)_3\exp_{U(1)}(-tix) \nonumber
\\ 
&=& 
-\bar{B}^T\frac{1}{i\sqrt{x^2+4B\bar{B}^T}}
e^{-\frac{ti}{2}(\sqrt{x^2+4B\bar{B}^T}+x)}\left(e^{ti\sqrt{x^2+4B\bar{B}^T}}-1\right).
\end{eqnarray}
It follows that $\gamma_{v}^3(t)=0$ for $t=\frac{2\pi}{\sqrt{x^2+4B\bar{B}^T}}$. That means that the geodesic reaches the set $L$ 
at $t=\frac{2\pi}{\sqrt{x^2+4B\bar{B}^T}}$, and since $L\subset K_{Id}$ it is also reaches the cut locus.   
This implies that the geodesic loses its optimality at the latest $t=\frac{2\pi}{\sqrt{x^2+4B\bar{B}^T}}$.

Having exact formulas for coordinates of geodesics we proceed to the core of the proof. Suppose $q\in V_{n,1}\setminus L$ but $q\in K_{\Id}$, and there exist two optimal normal geodesics $\gamma_{v_1}$ and $\gamma_{v_2}$ with 
$\gamma_{v_1}(0)=\gamma_{v_2}(0)=\Id$, $\gamma_{v_1}(T^*)=\gamma_{v_2}(T^*)=q$ and 
$v_1=\begin{pmatrix} x_1i & B \\ -\bar{B}^T & 0 \end{pmatrix}$, $v_2=\begin{pmatrix} x_2i & E \\ -\bar{E}^T & 0 \end{pmatrix}$ such that  $v_1\not = v_2$ and 
$x_{j} \in \mathbb{R}$, $j=1,2$ and $B,E \in \mathbb{C}^{1 \times (n-1)}$. 

Further we argue mostly as in the proof of Lemma~\ref{cutloci12}.
% and show that $T^*$ can not satisfies $0<T^* < \frac{2\pi}{\sqrt{x^2+4B\bar{B}^T}}$.   

{\textbf{Case 1.}} Assume $B\bar{B}^T \neq E\bar{E}^T$.  
Since the length of both geodesics $\gamma_{v_1}$ and $\gamma_{v_2}$ should coincide, it 
can be shown, as in the proof of Proposition~\ref{length}, that 
$$ B\bar{B}^T= \tr(B\bar{B}^T) =\| B \| ^2 = \Vert E \Vert^2= \tr(E\bar{E}^T)=E \bar{E}^T$$
This is a contradiction. 

{\textbf{Case 2.}} Let $x_1 = x_2$ and $\|B\|^2 = \|E\|^2$. 
 It follows from 
\begin{eqnarray*}
\gamma_{v_1}^3(T^*)&=&\gamma_{v_2}^3(T^*) \qquad\Longleftrightarrow\qquad
\\ 
&-&
\bar{B}^T\frac{1}{i\sqrt{x_1^2+4B\bar{B}^T}}e^{-\frac{iT}{2}(\sqrt{x_1^2+4B\bar{B}^T}+x_1)}(e^{iT\sqrt{x_1^2+4B\bar{B}^T}}-1)  
\\ 
=&-&\bar{E}^T\frac{1}{i\sqrt{x_2^2+4E\bar{E}^T}}e^{-\frac{iT}{2}(\sqrt{x_2^2+4E\bar{E}^T}+x_2)}(e^{iT\sqrt{x_2^2+4E\bar{E}^T}}-1)
\end{eqnarray*}
that $\bar{B}^T=\bar{E}^T$ and so $B=E$, that leads to the contradiction with $v_1\not=v_2$.

{\textbf{Case 3.}} Let now $x_1\not=x_2$ and $\|B\|^2 = \|E\|^2$. 
We know that $\gamma^3_{v_1}(T^*)=\gamma^3_{v_2}(T^*)\not=0$, which implies $\Vert \gamma^3_{v_1}(T^*) \Vert=\Vert \gamma^3_{v_2}(T^*) \Vert$. Thus
\begin{eqnarray*}
\frac{\Vert B \Vert}{\sqrt{x_1^2+4B\bar{B}^T}} \Vert 2\sin(\frac{T^*}{2}\sqrt{x_1^2+4B\bar{B}^T}) \Vert =
\frac{\Vert E \Vert}{\sqrt{x_2^2+4E\bar{E}^T}} \Vert 2\sin(\frac{T^*}{2}\sqrt{x_2^2+4E\bar{E}^T}) \Vert,
\end{eqnarray*}
and we get a contradiction as was shown in Case 3 of the proof of Lemma~\ref{cutloci12}.
\end{proof}

\begin{rem}
We observed that the distribution $\mathcal H$ is strongly bracket generating. It may be worth mentioning that $V_{n,1}$ is also a contact manifold, which was studied in~\cite{Go} and also in~\cite{M}. To show that statement, we note that the submersion $U(1) \rightarrow V_{n,1} \rightarrow Gr_{n,1}$ can be written as $S^1 \rightarrow S^{2n-1} \rightarrow \mathbb{C}P^{n-1}$. In~\cite{Go} it is shown that for submertion
$S^{2n-1} \rightarrow \mathbb{C}P^{n-1}$ the vertical vector space is spanned by
$$V(q)=-y_0\partial_{x_0}+x_0\partial_{y_0}-\dotso -y_{n-1}\partial_{x_{n-1}}+x_{n-1}\partial_{y_{n-1}},\quad q\in S^{2n-1}.
$$
The horizontal distribution $D$ is defined as the orthogonal complement to $\spn\{V\}$ in $TS^{2n-1}$ with respect to the Euclidean metric in $\mathbb R^{2n}\cong \mathbb C^n$. At the point $(1,0,\ldots,0)\in S^{2n-1}$ the vertical vector $V=(i,0,\ldots,0)$ coincides with the generator $\xi=\begin{pmatrix} i & 0 \\ 0 & 0 \end{pmatrix}$ of the Lie algebra $\mathfrak u_n(1)$ and the horizontal distribution $D=V^{\bot}$ coincides with the horizontal distribution $\mathcal H=\left\{\begin{pmatrix} 0 & B \\ -\bar{B}^T & 0 \end{pmatrix}\mid\ B\in\mathbb C^{1\times(n-1)}\right\}$ that is orthogonal to $\xi$ with respect to the trace metric. Since metrics, vertical and horizontal distributions are invariant under the action of $U(n)$ we conclude that sub-Riemannian geometries are essentially the same. It is shown in~\cite{Go} that the distribution $D$ coincides with the holomorphic tangent space $HS^{2n-1}$ of $S^{2n-1}$ thought of as an embedded CR-manifold and that it also coincides with the contact distribution given by $\ker (\omega)$ with respect to the contact form 
$$\omega=-y_0dx_0+x_0dy_0-\dotso -y_{n-1}dx_{n-1}+x_{n-1}dy_{n-1}.$$
Thus the contact structure can be transferred to the Stiefel manifold.
\end{rem}

%%%%%%%%%%%%%%%%%%%%%%%%%%
%%%%%%%%%%%%%%%%%%%%%%%%%%

\section{The cut loci of $V_{n,k}$}\label{sec:general}

%%%%%%%%%%%%%%%%%%%%%%%%%%
%%%%%%%%%%%%%%%%%%%%%%%%%%

In the present section we show that some of the properties of the cut locus of $V_{n,1}$ is preserved in the case $V_{n,k}$. In general we were not able to describe the total cut locus, since the distribution is not always strongly bracket generating, that leads to the existence of abnormal minimizers. Abnormal minimizers are also have to be taken into account since they can be minimizers due to~\cite{M94}. The interested reader can find a further information about abnormal minimizers in~\cite{AS,BC,BCK,BT,CM98,CM97,HY,LS}.

The fact that the distribution is in general not strongly bracket generating follows from the following proposition in~\cite{M}.
\begin{prop}
Let $Q$ be an $m$-dimensional manifold and $\mathcal{H}$ an $l$-dimensional strongly bracket generating distribution of codimension $2$ or greater. Then at least one of the following conditions 
$$(1)\ \ l\quad\text{is a multiple of}\quad 4,\qquad (2)\ \ l \geq (m-l)+1.$$
have to be fulfilled. 
\end{prop} 
It is obvious that it is not always the case for $V_{n,k}$, where $m=2nk-k^2$ and $l=2nk-2k^2$. Moreover, it is technically hard to write the exact form of normal sub-Riemannian geodesics for an arbitrary $V_{n,k}$.

\begin{prop}
The distribution $\mathcal H$ on $V_{n,k}$ is bracket generating.  
\end{prop}
\begin{proof}
%To show the proposition we define the $k\times (n-k)$ matrices 
%$$B_l(i):= (b_{\alpha\beta}) \in \mathbb{C}^{k \times (n-k)}\quad\text{with}\quad b_{\alpha\beta}= \begin{cases} i & \text{for } \alpha=\beta=l \\ 0 & \text{otherwise} \end{cases},
%$$ 
%$$C_l:=(c_{\alpha\beta}) \in \mathbb{C}^{k \times (n-k)}\quad\text{with}\quad c_{\alpha\beta}= \begin{cases} 1 & \text{for } \alpha=\beta=l \\ 0 & \text{otherwise} \end{cases},
%$$ 
%and 
%$$B_{sd}:=(b_{\alpha\beta}) \in \mathbb{C}^{k \times (n-k)}\quad\text{with}\quad b_{\alpha\beta}= \begin{cases} 1+i & \text{for } \alpha=s, \beta=d \\ 0 & \text{otherwise} \end{cases},
%$$ 
%$d \geq s+1$, $s \in \{1, \dotso k-1\}$, $d \in \{2, \dotso , k\}$. 
%We also write
%$$D_{sd}:=(d_{\alpha\beta})\in \mathbb{C}^{k \times k}\quad\text{with}\quad D_{sd}=0\quad\text{for}\quad d>k
%$$ 
%and 
%$$d_{\alpha\beta}= \begin{cases} -1-i & \text{for } \alpha=s, \beta=d \\ 1-i & \text{for } \alpha=d, \beta=s \\0 & \text{otherwise} \end{cases}\quad\text{for}\quad d \leq k,
%$$ 
%and
%$$D_l(i):=(d_{\alpha\beta})\in \mathbb{C}^{k \times k}\quad\text{with}\quad d_{\alpha\beta}= \begin{cases} -2i & \text{for } \alpha=\beta=l  \\0 & \text{otherwise} \end{cases}.
%$$
%It is clear that $D_{sd}$ and $D_l(i)$ form a bases for $\mathfrak{u}(k)$ for $l \in \{1, \dotso k\}$ and $d \geq s+1$, $s \in \{1, \dotso k-1\}$, $d \in \{2, \dotso , k\}$. It follows that 
%\begin{eqnarray*} 
%-B_l(i) \bar{C_l}^T + C_l\bar{B}_l(i)^T &=& D_l(i) \\
%-B_{sd} \bar{C_l}^T + C_l\bar{B}_{sd}^T &=& D_{sd}.
%\end{eqnarray*}
First we note that the commutator $[\mathcal H \,, \mathcal H]$ is given by 
\begin{eqnarray*}
\left[ \begin{pmatrix} 0 & B \\ -\bar{B}^T & 0 \end{pmatrix} \,, \begin{pmatrix} 0 & C \\ -\bar{C}^T & 0 \end{pmatrix} \right] &=& \begin{pmatrix} -B\bar{C}^T + C\bar{B}^T & 0 \\ 0 & -\bar{C}^TB + \bar{B}^TC \end{pmatrix}.
\end{eqnarray*}
It is sufficient if we show that for every upper triangular $(k\times k)$-matrix $D_{lm}$, $m>l$ with an entry $d_{lm}\neq 0$ on the intersection of $l$-row and $m$-column and all other entries vanish we can find $B,C \in \mathbb{C}^{k \times (n-k)}$ such that $D_{lm}=-B\bar{C}^T$. For instance, if we choose
$$B=\big(b_{\alpha \beta}\big)=\quad\text{by}\quad b_{\alpha \beta}= \begin{cases} d_{lm} & \text{for } \alpha=l, \beta=\min\{m , n-k\} \\ 0 & \text{otherwise}, \end{cases}
$$
and
$$-C^{T}=(c_{\alpha \beta})\quad\text{by}\quad c_{\alpha \beta}= \begin{cases} 1 & \text{for } \alpha=\min\{m , n-k\}, \beta=m \\ 0 & \text{otherwise}, \end{cases}
$$
then we deduce that $D_{lm}=-B\bar{C}^T$. 

We also need to construct diagonal-form $(k\times k)$-matrices $D_j$ with $i\in\mathbb C$ on the intersection of $j$-row and $j$-column and all other entries vanish and show that there are $B,C \in \mathbb{C}^{k \times (n-k)}$ such that $D_{j}=-B\bar{C}^T$. In this case we choose
$$B=(b_{\alpha \beta})\quad\text{by}\quad b_{\alpha \beta}= \begin{cases} i & \text{for } \alpha=j, \beta=\min\{j , n-k\} \\ 0 & \text{otherwise}, \end{cases}
$$ 
and
$$-\bar{C}^T=(c_{\alpha \beta})\quad\text{by}\quad c_{\alpha \beta}= \begin{cases} 1 & \text{for } \alpha=\min\{j , n-k\}, \beta=j \\ 0 & \text{otherwise}. \end{cases}$$ 
Then we obtain that $D_{j}=-B\bar{C}^T$. 
It implies that $\mathcal H$ is bracket generating of step $2$.
\end{proof}

\begin{prop}\label{length}
Suppose $\gamma_v(t)$ is a normal sub-Riemannian geodesic, which connects the identity $\Id$ with the point $q\in V_{n,k}$, $q \neq Id$, at the time
$T>0$, where $v= \begin{pmatrix} A & B \\ - \bar{B}^T & 0 \end{pmatrix} \in \mathfrak{u}(n)$.
The length of $\gamma_v$ is given by
$l(\gamma_v , T)=2T\sqrt{n \tr ( B\bar{B}^T)}$.
\end{prop}

\begin{proof}
First of all we calculate the velocity vector of $\gamma_v(t)$ at the time $t$, which is $\dot{\gamma}_v(t)=\gamma_v(t)
v_D$ for $v_D \in \mathfrak{u}(n)$. We omit the subscript $U(n)$ or $U(k)$ from $\exp_{(\cdot)}$, since it is clear which one we use from the context. By the chain rule we get that 
\begin{eqnarray*}
\dot{\gamma}_v(t) 
&=& 
d_{p(t)}\pi_1 
\bigg[\left(\exp \left\{t \begin{pmatrix} A & B \\ -\bar{B}^T & 0 \end{pmatrix} \right\} \right) 
\begin{pmatrix} A & B \\ -\bar{B}^T & 0 \end{pmatrix} \left(\exp \left\{t
\begin{pmatrix} -A & 0 \\ 0 & 0 \end{pmatrix} \right\} \right) 
\\
&+& 
\left(\exp \left\{t \begin{pmatrix} A & B \\ -\bar{B}^T & 0 \end{pmatrix} \right\} \right) \left(\exp \left\{t\begin{pmatrix} -A & 0 \\ 0 & 0 \end{pmatrix}  \right\} \right) \begin{pmatrix} -A & 0 \\ 0 & 0
\end{pmatrix}
\bigg],
\end{eqnarray*}
where $p(t):= \exp \left(t \begin{pmatrix} A & B \\ -\bar{B}^T & 0 \end{pmatrix} \right) \exp \left(t\begin{pmatrix} -A & 0 \\ 0 & 0 \end{pmatrix} \right)$.
We note that
\begin{eqnarray*}
&\begin{pmatrix} A & B \\ -\bar{B}^T & 0 \end{pmatrix} 
\exp \left\{t \begin{pmatrix} -A & 0 \\ 0 & 0 \end{pmatrix} \right\} 
= 
\begin{pmatrix} A \exp (-tA) & B
\\ -\bar{B}^T\exp (-tA) & 0 
\end{pmatrix} 
\\
&= 
\exp \left\{t \begin{pmatrix} -A & 0 \\ 0 & 0 \end{pmatrix} \right\} 
\begin{pmatrix} \exp (tA) A \exp (-tA) & \exp (tA) B \\ - \bar{B}^T \exp (-tA) & 0 \end{pmatrix}
\\
&= 
\exp \left\{t \begin{pmatrix} -A & 0 \\ 0 & 0 \end{pmatrix} \right\}  
\begin{pmatrix}  A  & \exp (tA) B \\ - \bar{B}^T \exp (-tA) & 0 
\end{pmatrix}.
\end{eqnarray*}
Thus 
\begin{eqnarray*}
\dot{\gamma}_v(t) 
&=& 
d_{p(t)}\pi_1 \bigg[ 
\exp \left\{t \begin{pmatrix} A & B \\ -\bar{B}^T & 0 \end{pmatrix} \right\} 
\exp \left\{t\begin{pmatrix} -A & 0 \\ 0 & 0 \end{pmatrix} \right\} 
\\
&\times&
\left(
\begin{pmatrix} A & \exp (tA) B \\ - \bar{B}^T \exp (-tA) & 0 \end{pmatrix} + \begin{pmatrix} -A & 0 \\ 0 & 0 \end{pmatrix} 
\right)
\bigg] 
\\
&=& \gamma_v(t) \begin{pmatrix}  0  & \exp (tA) B \\ - \bar{B}^T \exp (-tA) & 0 \end{pmatrix} 
\end{eqnarray*}
and 
$$
v_D= \begin{pmatrix}  0  & \exp (tA) B \\ - \bar{B}^T \exp (-tA) & 0 \end{pmatrix}.
$$
It follows that
\begin{eqnarray*}
g(\dot{\gamma}_v(t) , \dot{\gamma}_v(t) ) &= & -2n \tr (v_D^2) 
\\
&=& -2n \tr \left( \begin{pmatrix} -\exp (tA) B\bar{B}^T \exp (-tA)  & 0 \\ 0 & -\bar{B}^TB \end{pmatrix} \right) 
\\
&=& -2n\Big( -\tr \big(\exp (tA) B\bar{B}^T \exp (-tA)\big) - \tr (\bar{B}^TB ) \Big) 
\\
&=& 4n \tr ( B\bar{B}^T ).
\end{eqnarray*}
In the last equation we used $\tr (XY)=\tr(YX)$ and $\tr(-X)=-\tr (X)$.

We conclude that the length of $\gamma_v$ does not depend on $A$, but depend on final time $T$ and the trace of the matrix $B \bar{B}^T$.
\end{proof}

\begin{theorem}\label{cutlocikn}
The set 
$$L=\left\{\left[\begin{pmatrix} C & 0 \\ 0 & D \end{pmatrix}\right]_{V_{n,k}} \Big\vert\ C \in U(k),\ D \in U(n-k) \}\setminus \left[\Id\right]_{V_{n,k}}\right\}
$$ 
is a subset of the cut locus $K_{\Id}$ on $V_{n,k}$.
\end{theorem}
\begin{proof}
Suppose the point $[g]_{V_{n,k}} = \left[\begin{pmatrix} C & 0 \\ 0 & D \end{pmatrix}\right]_{V_{n,k}} \in L$. This point is reached optimally by a geodesic $\gamma_v=\left[\begin{pmatrix}
\gamma_v^1 & \gamma_v^2 \\ \gamma_v^3 & \gamma_v^4 \end{pmatrix}\right]_{V_{n,k}}$ at some $T$ from the initial point $\Id \in V_{n,k}$ and the initial velocity vector $v=
\begin{pmatrix} A & B \\ -\bar{B}^T & 0 \end{pmatrix} \in \mathfrak{u}(n)$. Let us see how $\gamma_v^j$, $j=1,2,3,4$, depend on $A$ and $B$. We recall that
\begin{eqnarray*}
\gamma_v(t)
&=&\pi_1\left(\exp \left\{t \begin{pmatrix} A & B \\ -\bar{B}^T & 0 \end{pmatrix} \right\} 
\exp \left\{t \begin{pmatrix} -A & 0 \\ 0 & 0 \end{pmatrix} \right\}\right) 
\\ 
&=& \left[\begin{pmatrix}
\gamma_v^1(t) & \gamma_v^2(t) \\ \gamma_v^3(t) & \gamma_v^4(t) \end{pmatrix}\right]_{V_{n,k}}.
\end{eqnarray*}
We start from calculating $\exp \left(t \begin{pmatrix} A & B \\ -\bar{B}^T & 0 \end{pmatrix} \right)=\begin{pmatrix} v_1(t) & v_2(t) \\ v_3(t) & v_4(t) \end{pmatrix} $. Using the notation $\begin{pmatrix} A & B \\ -\bar{B}^T & 0 \end{pmatrix}^n:=
\begin{pmatrix} v_1(n) & v_2(n) \\ v_3(n) & v_4(n) \end{pmatrix}$, we find that $v_1(n)=v_1(n-1)A-v_1(n-2)B\bar{B}^T$, $n\geq 2$, for initial values
$v_1(0)=Id$ and $v_1(1)=A$. This implies that $v_1$ as function of $t$ depends on $A$ and $B\bar{B}^T$. 
Furthermore, we get the formulas $v_2(n)=v_1(n-1)B$, $v_3(n)=-\bar{B}^Tv_1(n-1)$ and $v_4(n)=-\bar{B}^Tv_1(n-2)B$.  

Now we claim that the geodesic $\gamma_{v^*}$ with $v^*:= \begin{pmatrix} A & -B \\ \bar{B}^T & 0 \end{pmatrix}$ is also  minimizing from $\Id$ to $[g]_{V_{n,k}}$ such that
with $\gamma_{v^*}(T)=[g]_{V_{n,k}}$. Indeed, 
since $(-B)(-\bar{B}^T)=B\bar{B}^T$ and $(-\bar{B}^T)(-B)=\bar{B}^TB$ the length of both geodesics coincides. It remains to show that
$\gamma_{v^*}(T)=[g]_{V_{n,k}}$. Observe, that the value $v_1(t)$ depends on $A$, $B\bar{B}^T$ and $t$, and therefore $\gamma_{v^*}^1(T)=\gamma_v^1(T)$. 
%As
%$v_4^n=-\bar{B}^Tv_1^{n-2}B=\bar{B}^T(v^*)_1^{n-2}(-B)=(v^*)_4^n$, $v_1^0=(v^*)_1^0$ and $v_1^1=(v^*)_1^1$ it follows that
%$\gamma_v^4(T)=\gamma_{v^*}^4(T)$. \\
Finally $\gamma_v^2(T)=\gamma_v^3(T)=0$ implies $\gamma_{v^*}^2(T)=-\gamma_v^2(T)=0=\gamma_v^2(T)$ and
$\gamma_{v^*}^3(T)=-\gamma_v^3(T)=0=\gamma_v^3(T)$. We conclude that $\gamma_{v^*}(T)=\gamma_v(T)$ and that $L\subset K_{\Id}$.

The geodesic $\gamma_{v^*}$ can be replaced by $\gamma_{\hat v}$ with $\hat v= \begin{pmatrix} A & -BU \\ (\overline{BU})^T & 0 \end{pmatrix}$ for all $U \in U(n-k)$. It is also a minimizing geodesic from $\Id$ to $[g]_{V_{n,k}}$,
with $\gamma_{\hat v}(T)=[g]_{V_{n,k}}$.
\end{proof}

%%%%%%%%%%%%%%%%%%%%%%%%%%%%%%%%%%%%%%%%%%%%%%%%%%%%%%%%

\begin{subsection}{Points that are not in the cut locus of $V_{2k,k}$}

%%%%%%%%%%%%%%%%%%%%%%%%%%%%%%%%%%%%%%%%%%%%%%%%%%%%%%%%%%%%%%%%%%

Since the description of the cut locus for general Stiefel manifolds is very complicated we focus on the Stiefel manifolds $V_{n,k}$  with $n=2k$ and study points which never can belong to cut locus. The main result of this section is the following.

\begin{prop}\label{2kk}
All the points of the form $\left[\begin{pmatrix} 0 & D \\ C & 0 \end{pmatrix}\right]_{V_{2k,k}}$ with $C,D \in U(k)$ are not in the normal cut locus of $V_{2k,k}$.
\end{prop}
We start the proof of Proposition~\ref{2kk} from the following lemma.
\begin{lemma}\label{grgeo}
The points $\left [\begin{pmatrix} 0 & D \\ C & 0 \end{pmatrix}\right]_{G_{2k,k}}\in Gr_{2k,k}$ is reached by geodesics starting from $\left [\begin{pmatrix} Id_k & 0 \\ 0 & Id_k \end{pmatrix}\right]_{G_{2k,k}}$ only if the initial velocity vector $v$ has the form $v=\begin{pmatrix} 0 & B \\ -\bar{B}^T & 0 \end{pmatrix}$ with $B \in U(k)$. If we assume that $\tr(B\bar{B}^T)=1$ the condition $B \in U(k)$ is changed to $\sqrt{k}B \in U(k)$.
\end{lemma}
\begin{proof}
Geodesics of the Grassmann manifold $Gr_{2k,k}$ are given by
\begin{equation}\label{eq:geodGr}
 \gamma_v(t)= \left[ \exp\left(t \begin{pmatrix} 0 & B \\ -\bar{B}^T & 0 \end{pmatrix}\right)\right]_{Gr_{2k,k}} =\begin{pmatrix}\gamma_v^1(t)&\gamma_v^2(t)
 \\
 \gamma_v^3(t) & \gamma_v^4(t)\end{pmatrix},
\end{equation}
where
\begin{eqnarray*}
\gamma_v^1(t)&=&\cos(t\sqrt{B\bar{B}^T}) \\
\gamma_v^3(t)&=&-\bar{B}^T\sin(t\sqrt{B\bar{B}^T})(\sqrt{B\bar{B}^T})^{-1}.
\end{eqnarray*}
We are looking for all geodesics for which there exists $T_0>0$, such that $\gamma_v^1(T_0)=0$ and $\gamma_v^3(T_0)=C$. As $C\in U(k)$ and particularly is invertible it follows from the form of $\gamma_v^3(T_0)$ that $B$ is invertible. 
Therefore, the matrix $B\bar{B}^T$ is positive definite and diagonalizable: $B\bar{B}^T=PDP^{-1}$, where $D=\diag(d_1, \dotso ,d_k)$ is a diagonal matrix with $d_i>0$ for $i \in \{1,\dotso,k\}$. This implies that 
$$\cos(t\sqrt{B\bar{B}^T})=P\cos(t\sqrt{D})P^{-1}$$ 
and so $\gamma_v^1(T_0)=\cos(T_0\sqrt{B\bar{B}^T})=0$ if and only if $\cos(T_0\sqrt{d_1})=\dotso = \cos(T_0\sqrt{d_k})=0$.

For the moment we assume that $B\in U(k)$. Then using the normalisation $\tr(B\bar{B}^T)=1$ we get $\sqrt{k}B \in U(k)$. Thus $B\bar{B}^T=\frac{1}{k}\Id_k=\diag(\frac{1}{k},\dotso,\frac{1}{k})$, and $T_0:=\min\{t>0 \vert \cos(t\sqrt{B\bar{B}^T})=0\}=\frac{\pi \sqrt{k}}{2}$. 

Now we want to show that no other minimizing geodesics exist except of those with the initial velocity defined by matrices from $U(k)$.  
Let $B$ be an arbitrary invertible matrix, not necessary from $U(k)$. If we again assume the normalization $\tr(B\bar{B}^T)=1$, then we get that there exist at least two eigenvalues $\frac{1}{\lambda_1}$ and $\frac{1}{\lambda_2}$ of $B\bar{B}^T$ with $0<\frac{1}{\lambda_1}<\frac{1}{k}<\frac{1}{\lambda_2}$. It follows that if $\cos(T_0\sqrt{B\bar{B}^T})=0$, then $\cos(\frac{T_0}{\sqrt{\lambda_1}})=0$. We conclude that $T_0\geq \frac{\pi \sqrt{\lambda_1}}{2}>\frac{\pi \sqrt{k}}{2}$ and a geodesic with initial velocity defined by the matrix $B$ and that reach the point $\left [\begin{pmatrix} 0 & D \\ C & 0 \end{pmatrix}\right]_{G_{2k,k}}$ at time $T_0$ is not minimizing. 
\end{proof}
\begin{coro}
Let $p=\left [\begin{pmatrix} 0 & D \\ C & 0 \end{pmatrix}\right]_{V_{2k,k}}\in V_{2k,k}$ with $C,D \in U(k)$ and $v=\begin{pmatrix} 0 & B \\ -\bar{B}^T & 0 \end{pmatrix}$ with $\sqrt{k}B \in U(k)$, $\tr(B\bar{B}^T)=1$. Then geodesics $\gamma_v(t)$ in $V_{2k,k}$ reaching the points $p$ at time $T_0=\frac{\pi \sqrt{k}}{2}$ are minimizing. Furthermore, if $B_1\not=B_2$, then $\gamma_{v_1}^3(T_0)\not=\gamma_{v_2}^3(T_0)$.  
\end{coro}

\begin{proof}
First we note that geodesics in $Gr_{2k,k}$ defined by $v$ satisfying the assumption of Lemma~\ref{grgeo} are minimizing geodesics from $\left [\begin{pmatrix} Id_k & 0 \\ 0 & Id_k \end{pmatrix}\right]_{G_{2k,k}}$ to $\left [\begin{pmatrix} 0 & D \\ C & 0 \end{pmatrix}\right]_{G_{2k,k}}$ by Lemma~\ref{grgeo}. The time of reaching the points $\left [\begin{pmatrix} 0 & D \\ C & 0 \end{pmatrix}\right]_{G_{2k,k}}$ is $T_0=\frac{\pi \sqrt{k}}{2}$. 
Furthermore, 
\begin{equation}\label{eq:g3}
\gamma_v^3(T_0)=-\bar{B}^T\diag\left(\sin(\frac{T_0}{\sqrt{k}}),\dotso,\sin(\frac{T_0}{\sqrt{k}})\right)\sqrt{k}=-\sqrt{k}\bar{B}^T \in U(k).
\end{equation}
The unique horizontal lift of~\eqref{eq:geodGr} is a minimizing between fibers passing through 
$[\Id]_{V_{2k,k}}$ and $p$ and moreover they are geodesics since they are horizontal lifts of geodesics.
Fix a point $p_0$ at the fiber passing through $[\Id]_{V_{2k,k}}$. Then the unique horizontal lift $\gamma_v(t)_{V_{2k,k}}=[\exp(tv)]_{V_{2k,k}}$ of~\eqref{eq:geodGr} starting from $p_0$ always reaches different points at the fiber $\pi^{-1}\left(\left [\begin{pmatrix} 0 & D \\ C & 0 \end{pmatrix}\right]_{G_{2k,k}}\right)$ at the time $T_0$ since $\gamma_v^3(T_0)$ depends on $\bar{B}^T$ but not on $B\bar{B}^T$ as shows~\eqref{eq:g3}. 
\end{proof}

Now we prove Proposition~\ref{2kk}.

\begin{proof}
Let assume that a point $p=\left[\begin{pmatrix} 0 & D \\ C & 0 \end{pmatrix} \right]_{V_{2k,k}}$ belongs to the cut locus from the identity point in $V_{2k,k}$.
Let $v^*=\begin{pmatrix} A & B \\ -\bar{B}^T & 0\end{pmatrix}$ with $A\not =0$ be an initial velocity of a minimizing normal horizontal geodesic 
$$
\gamma^{*}(t)=\left[\exp\left(t\begin{pmatrix} A & B \\ -\bar{B}^T & 0 \end{pmatrix} \right)\right]_{V_{2k,k}}\exp\left(-t\begin{pmatrix} A & 0 \\ 0 & 0 \end{pmatrix} \right)
$$ 
from $[Id]_{V_{2k,k}}$ to $p$ such that $\gamma^{*}(T_0)=p$. Then its projection $\tilde\gamma$ to $Gr_{2k,k}$ is a minimizing geodesic from $[Id]_{Gr_{2k,k}}$ to $\left [\begin{pmatrix} 0 & D \\ C & 0 \end{pmatrix}\right]_{G_{2k,k}}$. 
This implies that $\tilde\gamma$ have to coincide with a geodesic in $Gr_{2k,k}$ having form~\eqref{eq:geodGr} with some $B_1$ satisfying $\sqrt{k}B_1 \in U(k)$. It is also clear that $\gamma^{*}(t)$ is a horizontal lift of $\tilde\gamma$ starting at the point $[Id]_{V_{2k,k}}$. From the other side the horizontal lift of a geodesic having form~\eqref{eq:geodGr} is equal to $\left[\exp\left(t\begin{pmatrix} 0 & B_1 \\ -\bar{B}_1^T & 0 \end{pmatrix} \right)\right]_{V_{2k,k}}$ which is different from $\gamma^{*}(t)$. This is a contradiction to the fact that horizontal lift starting from the same point is unique. We conclude that the points of the form $\left[\begin{pmatrix} 0 & D \\ C & 0 \end{pmatrix} \right]_{V_{2k,k}}$ can not be in the cut locus.
\end{proof}

\begin{coro}
The set 
$$ \left\{\left[ \exp\left(t \begin{pmatrix} 0 & B \\ -\bar{B}^T & 0 \end{pmatrix}\right)\right]_{V_{2k,k}} \Big\vert \tr(B\bar{B}^T)=1 \,, \sqrt{k}B \in U(k) \,, t \in \left[0 \,, \frac{\pi \sqrt{k}}{2}\right] \right\} $$
is not in the normal cut locus of $V_{2k,k}$.
\end{coro} 

\end{subsection}

%%%%%%%%%%%%%%%%%%%%%%%%%%%%%%%

\begin{section}{Stiefel and Grassmann manifold as embedded into $SO(n)$}\label{real}
In this section we assume that the Stiefel and Grassmann manifolds are embedded into $SO(n)$. We use similar notations for the Stiefel and the Grassmann manifolds as in the previous sections. 

%%%%%%%%%%%%%%%%%%%%%%%%%%%%%%%

\subsection{The group $SO(n)$, Stiefel and Grassmann manifolds} 

%%%%%%%%%%%%%%%%%%%%%%%%%%%%%%%

We recall that the special orthogonal group $SO(n)$ is the set of matrices
$$SO(n):=\{X \in \mathbb{R}^{n\times n} \vert\  X^TX=XX^T=I_n \text{ },\text{} \det(X)=1\}.$$
This is a compact Lie group with the Lie algebra $\mathfrak{so}(n)$ given by
$$  \mathfrak{so}(n) := \{ \mathcal{X} \in \mathbb{R}^{n \times n} \vert\  \mathcal X = -\mathcal X^T \}.$$
Every entry on the diagonal of $\mathcal X \in \mathfrak{so}(n)$ is zero and the real dimension of the manifold is $\frac{1}{2}n(n-1)$.

We define a bi-invariant Riemannian metric on $SO(n)$ by 
\begin{eqnarray*}
\langle \cdot \,, \cdot \rangle \colon q \mathfrak{so}(n) \times q \mathfrak{so}(n) &\to& \mathbb{R} \\
\langle q \mathcal{X} \,, q \mathcal{Y} \rangle &:=& -\tr (\mathcal{X}\mathcal{Y}),
\end{eqnarray*}
with $\mathcal{X},\mathcal{Y} \in \mathfrak{so}(n)$.

The Stiefel manifold $V_{n,k}$ for $k<n$ is the set of all $k$-tuples $(q_1, \dotso , q_k)$ of vectors $q_i \in \mathbb{R}^n$, $ i \in \{1, \dotso , k\}$, which are orthonormal with respect to the standard Euclidean metric. This compact manifold can be equivalently defined as 
$$
V_{n,k}:=\{X\in\mathbb{R}^{n\times k} \vert \ \ X^TX=I_{k}\}.
$$
Another way to define the Stiefel manifold $V_{n,k}$ is to introduce the equivalence relation $\backsim_1$ in $SO(n)$ by
\begin{eqnarray*}
q \backsim_1 p\quad\Longleftrightarrow\quad q=p \begin{pmatrix} I_k & 0 \\ 0 & S_{n-k} \end{pmatrix}, \qquad q,p \in SO(n), \quad  S_{n-k} \in SO(n-k),
\end{eqnarray*}
such that the equivalence class $[q]^{\backsim_1}$ of a point $q \in SO(n)$ is given by

\begin{eqnarray*}
[q]^{\backsim_1}=\left \{ q\begin{pmatrix} I_k & 0 \\ 0 & S_{n-k} \end{pmatrix} \Big  \vert S_{n-k} \in SO(n-k) \right \} \in SO(n)/SO_n(n-k).
\end{eqnarray*}
The Stiefel manifold $V_{n,k}$ can be identified with $SO(n)/SO_n(n-k)$. We use the notation $[q]_{V_{n,k}}$ for $[q]^{\backsim_1}$ in the present section.

The tangent space at a point $[q]_{V_{n,k}} \in V_{n,k}$ is given by
$$T_{[q]_{V_{n,k}}}V_{n,k}=\left\{ [q]_{V_{n,k}} \begin{pmatrix} \mathcal X_1 & -\mathcal X_2^T \\ \mathcal X_2 & 0 \end{pmatrix} 
\Big\vert \ \ \mathcal X_1 \in \mathfrak{so}(k), \mathcal X_2 \in \mathbb{R}^{(n-k) \times k} \right\} . $$
The induced metric
 $\langle \cdot\, , \cdot \rangle_{V_{n,k}}$ on $V_{n,k}$ is given by
\begin{eqnarray*}
& \left\langle  [q]_{V_{n,k}} 
\begin{pmatrix} \mathcal X_1 & -\mathcal X_2^T \\ \mathcal X_2 & 0 
\end{pmatrix} , 
 [q]_{V_{n,k}} 
 \begin{pmatrix} \mathcal Y_1 & -\mathcal Y_2^T \\ \mathcal Y_2 & 0 
 \end{pmatrix} 
 \right\rangle_{V_{n,k}} \Big( [q]_{V_{n,k}} \Big)
 \\
 &:=
 \left\langle q 
 \begin{pmatrix} \mathcal X_1 & -\mathcal X_2^T \\ \mathcal X_2 & 0 
 \end{pmatrix} , 
 q \begin{pmatrix} \mathcal Y_1 & -\mathcal Y_2^T \\ \mathcal Y_2 & 0 
 \end{pmatrix} 
 \right\rangle_{SO(n)} \big(q\big)
 %= \left( 
%\begin{pmatrix} \mathcal X_1 & -\bar{\mathcal X_2}^T \\ \mathcal X_2 & 0 
%\end{pmatrix} ,  
%\begin{pmatrix} \mathcal Y_1 & -\bar{\mathcal Y_2}^T \\ \mathcal Y_2 & 0 
%\end{pmatrix} 
%\right)_{\mathfrak{so}(n)} 
\\
&= - \tr \left ( \begin{pmatrix} \mathcal X_1 & -\mathcal X_2^T \\ \mathcal X_2 & 0 
\end{pmatrix}  
\begin{pmatrix} \mathcal Y_1 & -\mathcal Y_2^T \\ \mathcal Y_2 & 0 
\end{pmatrix} \right ) ,
\end{eqnarray*} 
where $q \in [q]_{V_{n,k}}$ is any representative of the equivalence class $[q]_{V_{n,k}}$. 

The Grassmann manifold $G_{n,k}$ is defined as a collection of all $k$-dimensional subspaces $\Lambda$ of $\mathbb{R}^n$. Equivalently, an element $\Lambda$ of $G_{n,k}$ can be written as an $(n \times k)$-matrix with columns $w_1, \dotso , w_k \in \mathbb{R}^n$, such that $\spn \{w_1, \dotso , w_k\}=\Lambda$, or, it can be defined as a quotient space in $SO(n)$ with respect to the following equivalence relation
\begin{eqnarray*}
m_1\backsim_2 m_2 \quad \Longleftrightarrow\quad m_1=m_2 \begin{pmatrix} S_k & 0 \\ 0 & S_{n-k} 
\end{pmatrix},\qquad m_1 , m_2 \in SO(n),
\end{eqnarray*}
where $S_k \in O(k)$, $S_{n-k} \in O(n-k)$, such that $\det(S_k)=\det(S_{n-k}) \in \{-1,1\}$. This leads to the equivalence classes
\begin{eqnarray*}
[m]^{\backsim_2}=\left\{m 
\begin{pmatrix} 
S_k & 0 
\\ 
0 & S_{n-k} 
\end{pmatrix} 
\Big\vert\ \  S_k \in O(k), S_{n-k} \in O(n-k) \text{ }, \text{} \det(S_k)=\det(S_{n-k}) \right\},  
\end{eqnarray*}
$m \in SO(n)$, which is isomorphic to $O(k) \times SO(n-k)\cong O_n(k) \times SO_n(n-k) $, as 
\begin{eqnarray*}
m\begin{pmatrix} 
S_k & 0 
\\ 
0 & S_{n-k} 
\end{pmatrix} \mapsto \left(S_k , S_{n-k} \begin{pmatrix} \det(S_k) & 0 \\ 0 & \Id_{n-1} \end{pmatrix} \right) \in O(k)\times SO(n-k) .
\end{eqnarray*}
We identify $G_{n,k}$ with the quotient space $SO(n)/(O_n(k)\times SO_n(n-k))$ and use the notation $[m]_{G_{n,k}}$ for $[m]^{\backsim_2}$ in the current section. 

The tangent space of $G_{n,k}$ at the point $[m]_{G_{n,k}}$ is given by
$$T_{[m]_{G_{n,k}}}G_{n,k}=\left\{[m]_{G_{n,k}} \begin{pmatrix} 0 & \mathcal X_2 \\ -\mathcal X_2^T & 0 \end{pmatrix} \Big\vert \ \ \mathcal X_2 \in \mathbb{R}^{k \times (n-k)} \right\} .$$ 
It has real dimension $k(n-k)$ that gives the real dimension of $G_{n,k}$. 

The induced metric $\langle \cdot\, , \cdot\rangle_{G_{n,k}}$ on $G_{n,k}$ is given by 
\begin{eqnarray*} 
& \left\langle
 [m]_{G_{n,k}} 
\begin{pmatrix} 0 & \mathcal X_2 \\ -\mathcal X_2^T & 0 
\end{pmatrix} , [m]_{G_{n,k}} 
\begin{pmatrix} 0 & \mathcal Y_2 \\ -\mathcal Y_2^T & 0 
\end{pmatrix} 
\right\rangle_{G_{n,k}} \Big([m]_{G_{n,k}} \Big)
\\
&:=  
\left\langle 
m 
\begin{pmatrix} 0 & \mathcal X_2 \\ -\mathcal X_2^T & 0 
\end{pmatrix} , m 
\begin{pmatrix} 0 & \mathcal Y_2 \\ -\mathcal Y_2^T & 0 
\end{pmatrix} 
\right\rangle_{SO(n)}\big(m\big) 
\\
&= 
-\tr \left( 
\begin{pmatrix} 0 & \mathcal X_2 \\ -\bar{\mathcal X_2}^T & 0 
\end{pmatrix} 
\begin{pmatrix} 0 & \mathcal Y_2 \\ -\bar{\mathcal Y_2}^T & 0 
\end{pmatrix} 
\right), 
\end{eqnarray*}
where $m \in SO(n)$ is any representative of $[m]_{G_{n,k}}$. 

A normal sub-Riemannian geodesic on $V_{n,k}$ starting from $ [q]_{V_{n,k}}$ is given by the formula similar to~\eqref{sRgeodesic} presented in Section~\ref{sec:submersion}.
\begin{eqnarray}\label{sRgeodesico}
\gamma (t) &=& \exp_{V_{n,k}}(tv) \exp_{O_n(k)}(-tA(v)) \nonumber
\\
&=&  \pi_1\left[q\exp_{SO(n)}\left(t 
\begin{pmatrix} \mathcal X_1 &  \mathcal X_2 \\ -\bar{ \mathcal X_2}^T & 0  
\end{pmatrix} \right)
\right] 
\exp_{O_n(k)}\left( -t \begin{pmatrix}  \mathcal X_1 & 0 \\ 0 & 0 \end{pmatrix} \right),
\end{eqnarray}
where $q \in SO(n)$, $v= [q]_{V_{n,k}} \begin{pmatrix}  \mathcal X_1 &  \mathcal X_2 \\ -\bar{ \mathcal X_2}^T & 0 \end{pmatrix} \in T_{[q]_{V_{n,k}} }V_{n,k}$ with $\begin{pmatrix}  \mathcal X_1 &  \mathcal X_2 \\ -\bar{ \mathcal X_2}^T
& 0 \end{pmatrix} \in \mathfrak{so}(n)$, $\pi_1\colon SO(n) \to SO(n)/SO_n(n-k)$ is the natural projection from $SO(n)$ to the quotient space, and $A\colon TV_{n,k} \to \mathfrak{so}_n(k)$ is the $\mathfrak{so}_n(k)$-valued connection one form.

%%%%%%%%%%%%%%%%%%%%%%%%%%%%%%%

\subsection{The cut locus of $V_{n,1}$, $n \in \mathbb{N}$} 

%%%%%%%%%%%%%%%%%%%%%%%%%%%%%%%
In this case $\dim(V_{n,1})=\dim(G_{n,1})=n-1$ and all sub-Riemannian geodesics are Riemannian ones. For the reason of completeness we present the cut locus in this case, because it is strongly related to the cut locus of $V_{n,1}$ embedded in $U(n)$. 

Two parts $\gamma_v^1(t), \gamma_v^3(t)$ of the geodesic $\gamma_v(t)=\left[ \begin{pmatrix} \gamma^1(t) & \gamma^2(t) \\ \gamma^3(t) & \gamma^4(t) \end{pmatrix}\right]_{V_{n,1}}$ for an initial velocity $v= \begin{pmatrix} 0 & B \\  -B^T & 0\end{pmatrix}$ are given by
\begin{eqnarray*}
\gamma_v^1(t)&=&\frac{1}{4\sqrt{BB^T}}e^{-it\sqrt{BB^T}}2\sqrt{BB^T}(e^{2it\sqrt{BB^T}}+1), \\
&=&\frac{1}{2}e^{-it\sqrt{BB^T}}(e^{2it\sqrt{BB^T}}+1)=\cos(t\sqrt{BB^T}) \\
\gamma_v^3(t) &=& -B^T\frac{1}{i2\sqrt{BB^T}}e^{-it\sqrt{BB^T}}(e^{2it\sqrt{BB^T}}-1) \\
&=& \frac{-B^T}{\sqrt{BB^T}}\sin(t\sqrt{BB^T}).
\end{eqnarray*}
These formulas are a particular case of formulas~\eqref{eq:g1} and~\eqref{eq:g2} for the choice of the initial velocity $v=\begin{pmatrix} 0 & B \\  -B^T & 0\end{pmatrix} \in \left\{\begin{pmatrix} xi & E \\ -\bar{E}^T & 0 \end{pmatrix} \Big\vert \ E \in \mathbb{C}^{1 \times (n-1)} , x \in \mathbb{R} \right\}$. 
Thus we can use arguments of Theorem~\ref{th:Vn1} and state that the cut locus of  the Stiefel manifold $V_{n,1}$ embedded in $SO(n)$ consists of exactly one point:
\begin{eqnarray*}
\left\{\left[\begin{pmatrix} C & 0 \\ 0 & D \end{pmatrix}\right]_{V_{n,k}} \Big\vert \ C \in O(1),\ D \in O(n-1):\begin{pmatrix} C & 0 \\ 0 & D \end{pmatrix} \in SO(n) \right\} \setminus  \left\{\left[Id\right]_{V_{n,k}}\right\} &=& \\
\left\{\left[\begin{pmatrix} \pm1 & 0 \\ 0 & D \end{pmatrix}\right]_{V_{n,k}} \Big\vert \  D \in O(n-1):\begin{pmatrix} \pm1 & 0 \\ 0 & D \end{pmatrix} \in SO(n) \right\} \setminus \left\{\left[Id\right]_{V_{n,k}}\right\}&=& \\ 
 \left\{\left[\begin{pmatrix} -1 & 0 \\ 0 & D \end{pmatrix}\right]_{V_{n,k}} \Big\vert \  D \in O(n-1):\begin{pmatrix} -1 & 0 \\ 0 & D \end{pmatrix} \in SO(n) \right\}.
\end{eqnarray*}

%%%%%%%%%%%%%%%%%%%%%%%%%%%%%%%

\subsection{The cut locus of $V_{3,2}$}

%%%%%%%%%%%%%%%%%%%%%%%%%%%%%%%

Since $V_{3,2} \cong SO(3)/SO(1)$ and $SO(1)$ is a normal subgroup of $SO(3)$, one can identify the sub-Riemannian structure of $V_{3,2}$ with the sub-Riemannian structure on the group $SO(3)$, that was studied in \cite{B}.
In particular all equivalences classes contain exactly one matrix
\begin{eqnarray*}
\left[\begin{pmatrix} a_{11} & a_{12} & a_{13} \\ a_{21} & a_{22} & a_{23} \\ a_{31} & a_{32} & a_{33} \end{pmatrix}\right]_{V_{n,k}} &=& 
\left\{\begin{pmatrix} a_{11} & a_{12} & a_{13} \\ a_{21} & a_{22} & a_{23} \\ a_{31} & a_{32} & a_{33} \end{pmatrix} \begin{pmatrix} I_2 & 0 \\ 0 & S_{1} \end{pmatrix} \Big\vert S_{1} \in S(1) \right \}  \\ &=&
\left \{\begin{pmatrix} a_{11} & a_{12} & a_{13} \\ a_{21} & a_{22} & a_{23} \\ a_{31} & a_{32} & a_{33} \end{pmatrix} \right \},
\end{eqnarray*}
such that we can identify $V_{3,2}$ with $SO(3)$. Furthermore, the induced horizontal and vertical space coincide with the horizontal and vertical space of the $k\oplus p$ problem on $SO(3)$ stated in \cite{B}. 

%%%%%%%%%%%%%%%%%%%%%%%%

\begin{subsection}{About the cut locus for $V_{2k,k}$}

%%%%%%%%%%%%%%%%%%%%%%%%%%%%%%%%%%%%%%%%%%%%%%%%%%%%%%%%%%%%%%%%%%

Analogous to Proposition~\eqref{2kk} we state here the following result.
\begin{prop}
All the points of the form $\left[\begin{pmatrix} 0 & D \\ C & 0 \end{pmatrix}\right]_{V_{2k,k}}$ with $C,D \in O(k)$ are not in the normal cut locus of $V_{2k,k}$.
\end{prop}

\begin{proof}
The proof of Proposition~\ref{2kk} does not use the specific of the unitary group, rather the orthogonality property. Therefore, we can literally repeat the proof of Proposition~\ref{2kk} here.
\end{proof}
\end{subsection}

\begin{coro}
The set 
$$ \left\{\left[ \exp\left(t \begin{pmatrix} 0 & B \\ -\bar{B}^T & 0 \end{pmatrix}\right)\right]_{V_{2k,k}} \Big\vert \tr(B\bar{B}^T)=1 \,, \sqrt{k}B \in O(k) \,, t \in \left[0 \,, \frac{\pi \sqrt{k}}{2}\right] \right\} $$
is not in the normal cut locus of $V_{2k,k}$.
\end{coro} 

\end{section}

%\bibliographystyle{amsalpha}
%\bibliography{Grassmann}

\begin{thebibliography}{99}

\bibitem{AgrSach}
{\sc Agrachev, A. A.; Sachkov, Y. L.}, {\it Control theory from the geometric viewpoint}. Encyclopaedia of Mathematical Sciences, 87. Control Theory and Optimization, II. Springer-Verlag, Berlin, 2004. pp. 412.

\bibitem{AS} {\sc Agrachev, A. A.; Sarychev, A. V.}, {\it Abnormal sub-{R}iemannian geodesics: {M}orse index and rigidity}. Ann. Inst. H. Poincar\'e Anal. Non Lin\'eaire. {\bf 13} (1996), 635-690.

\bibitem{AS1}
{\sc Agrachev, A. A.; Sarychev, A. V.}, {\it Sub-Riemannian metrics: minimality of abnormal geodesics versus subanalyticity}. ESAIM Control Optim. Calc. Var. {\bf 4} (1999), 377-403. 

\bibitem{cut} {\sc Agrachev, A.; Bonnard, B.; Chyba, M.; Kupka, I.}, {\it Sub-{R}iemannian sphere in {M}artinet flat case}. ESAIM Control, Optim. Calc. Var. {\bf 2} (1997), 377-448.

\bibitem{AG} {\sc Agrachev, A. A.; Gauthier, J. P.}, {\it Sub-Riemannian metrics and isoperimetric problems in the contact case.} (Russian) Geometric control theory (Russian) (Moscow, 1998), 5-48, Itogi Nauki Tekh. Ser. Sovrem. Mat. Prilozh. Temat. Obz., 64, Vseross. Inst. Nauchn. i Tekhn. Inform. (VINITI), Moscow, 1999.

\bibitem{A} {\sc Edelman, A.; Arias, T. A.; Smith, S. T.}, {\it The geometry of algorithms with orthogonality constraints}. SIAM J. Matrix Anal. Appl. {\bf 20} (1999), no.~2., 303-353.

\bibitem{BC} {\sc Bonnard, B.; Chyba, M.}, {\it M\'ethodes g\'eom\'etriques et analytiques pour \'etudier l'application exponentielle, la sph\`ere et le front d'onde en g\'eom\'etrie sous-riemannienne dans le cas {M}artinet}. ESAIM. Control Optim. Calc. Var. {\bf 4} (1999), 245-334.

\bibitem{BCK} {\sc Bonnard, B.; Chyba, M.; Kupka, I.}, {\it Nonintegrable geodesics in {SR}-{M}artinet geometry}. Differential geometry and control (Boulder, CO, 1997), 119-134, Proc. Sympos. Pure Math., 64, Amer. Math. Soc., Providence, RI, 1999. 

\bibitem{BT} {\sc Bonnard, B.; Tr{\'e}lat, E.}, {\it On the role of abnormal minimizers in sub-{R}iemannian geometry}. Ann. Fac. Sci. Toulouse Math. (6) {\bf 10} (2001), no.~3, 405-491.

\bibitem{B}
{\sc Boscain, U.; Rossi, F.}, {\it Invariant {C}arnot-{C}arath\'eodory metrics on {$S^3,\ {\rm SO}(3),\ {\rm SL}(2)$}, and lens spaces}. SIAM J. Control Optim. {\bf 47} (2008), no.~4, 1851-1878. 

\bibitem{CDPT}
{\sc Capogna, L.; Danielli, D.; Pauls, S. D.; Tyson, J.T.}, {\it An introduction to the Heisenberg group and the sub-Riemannian isoperimetric problem}. Progress in Mathematics, 259. Birkh\"auser Verlag, Basel, 2007. pp. 223.

\bibitem{C} {\sc Chow, W.-L.}, {\it \"Uber Systeme von linearen partiellen Differentialgleichungen erster Ordnung}. Math. Ann. {\bf 117} (1939), 98-105.

\bibitem{CM98} {\sc Chyba, M.}, {\it Le front d'onde en g\'eom\'etrie sous-riemannienne: le cas Martinet}. S\'eminaire de Th\'eorie Spectrale et G\'eom\'etrie, {\bf 16} (1997-1998),  pp. 81-105. Univ. Grenoble I, Saint.

\bibitem{CM97} {\sc Chyba, M.}, {\it La cas Martinet en G\'eom\'etrie Sous-Riemannienne}. Ph.D. thesis, Univ. of Geneva.

\bibitem{G} {\sc Gallier, J.}, {\it Notes on Differential Geometry and Lie Groups}. Unpublished manuscript.  Retrieved from http://www.cis.upenn.edu/~cis610/diffgeom-n.pdf

\bibitem{Go} {\sc Godoy Molina, M.; Markina, I.}, {\it Sub-Riemannian geodesics and heat operator on odd dimensional spheres}. Anal. Math. Phys. {\bf 2} (2012), no.~2, 123-147.

\bibitem{GR} {\sc Grochowski, M.}, {\it Normal forms and reachable sets for analytic {M}artinet sub-{L}orentzian structures of {H}amiltonian type}. J. Dyn. Control Syst. {\bf 17} (2011), no.~1, 49-75.

\bibitem{HY} {\sc Huang, T.; Yang, X.}, {\it Extremals in some classes of {C}arnot groups}. Sci. China Math. {\bf 55} (2012), no.~3, 633-646.

\bibitem{Knapp}
{\sc Knapp, A. W.}, {\it Lie groups beyond an introduction}. Second edition. Progress in Mathematics, 140. Birkh\"auser Boston, Inc., Boston, MA, 2002. pp. 812.

\bibitem{LS} {\sc Liu, W.; Sussmann, H. J.}, {\it Shortest paths for sub-Riemannian metrics on rank-two distributions}. Mem. Amer. Math. Soc. {\bf 118} (1995), pp. 104.

\bibitem{Ma} {\sc Manton, J. H.},  {\it Optimization algorithms exploiting unitary constraints}. IEEE Trans. on Signal Process. {\bf 50} (2002), no.~3, 635-650.

\bibitem{Mi} {\sc Milnor, J.}, {\it Curvatures of left invariant metrics on {L}ie groups}, Advances in Math. {\bf 21} (1976), no.~3, 293-329.

\bibitem{M94} {\sc Montgomery, R.}, {\it Abnormal minimizers}. SIAM J. Control Optim. {\bf 32} (1994), no.~6, 1605-1620.

\bibitem{M} 
{\sc Montgomery, R.}, {\it A tour of subriemannian geometries, their geodesics and applications}. Mathematical Surveys and Monographs, 91. American Mathematical Society, Providence, RI, 2002.

\bibitem{NaSa}
{\sc Naitoh, H.; Sakane, Y.}, {\it On conjugate points of a nilpotent Lie group}. Tsukuba J. Math. {\bf 5} (1981), no.~1, 143-152.

\bibitem{R} {\sc Rashevski{\u\i}, P. K.}, {\it About connecting two points of complete nonholonomic space by admissible curve}, Uch. Zapiski Ped. Inst. K.~Liebknecht {\bf 2} (1938), 83-94.

\bibitem{S} {\sc Sachkov, Y.}, {\it Cut locus and optimal synthesis in the sub-{R}iemannian problem on the group of motions of a plane}. ESAIM. Control Optim. Calc. Var. {\bf 17} (2011), no.~2, 293-321.

\bibitem{Str}
{\sc Strichartz, R. S.}, {\it Sub-Riemannian geometry}. J. Differential Geom. {\bf 24} (1986), no. 2, no.~2, 221-263.

\end{thebibliography}
\end{document}